\documentclass[12pt, reqno, a4paper]{amsart}
\usepackage[margin=1.2in]{geometry}
\numberwithin{equation}{section}
\usepackage{amssymb,amsfonts,amsthm}
\usepackage{mathtools}
\usepackage[utf8]{inputenc}
\usepackage{listings}
\usepackage{bm}
\usepackage[hyperpageref]{backref}
\usepackage{esint}
\usepackage{color}
\usepackage[bookmarks]{hyperref}
\usepackage{esint}
\usepackage{bigints}

\allowdisplaybreaks[4]

\newtheoremstyle{myremark}{10pt}{10pt}{}{}{\bfseries}{.}{.5em}{}

\newtheorem{theorem}{Theorem}[section]

\newtheorem{lemma}[theorem]{Lemma}

\theoremstyle{definition}

\newtheorem{remark}{Remark}

\newcommand{\meas}{\left|\mathbb{S}^{d-1}\right|}
\newcommand{\al}{\alpha}
\newcommand{\be}{\beta}
\newcommand{\ga}{\gamma}
\newcommand{\ve}{\varepsilon}
\newcommand{\R}{\mathbb{R}}
\usepackage{hyperref}

\allowdisplaybreaks[4]

\begin{document}

\title[Fractional Hardy--Sobolev inequalities]{Weighted fractional Hardy--Sobolev and Hardy--Sobolev--Maz'ya inequalities with singularities on flat submanifold}

\author{Michał Kijaczko and Vivek Sahu}

\address{Faculty of Pure and Applied Mathematics, Wrocław University of Science and Technology, Wybrzeże Wyspiańskiego 27, 50-370 Wrocław, Poland}
\email{michal.kijaczko@pwr.edu.pl}

\address{Department of Mathematics and Statistics,
Indian Institute of Technology Kanpur, Kanpur - 208016, Uttar Pradesh, India}
\email{viveksahu20@iitk.ac.in}



\date{}

\dedicatory{}

\begin{abstract}  
We investigate the sharp constant for weighted fractional Hardy inequalities with the singularity on a flat submanifold of codimension $k$, where $1 \leq k < d$. We also prove a weighted fractional Hardy inequality with a remainder. Using this result, we extend and derive a weighted version of the fractional Hardy–Sobolev–Maz'ya inequality with singularities on a flat submanifold. Furthermore, we obtain a weighted logarithmic fractional Hardy–Sobolev–Maz'ya inequality in the case of a singularity at the origin and we show that in this case, the fractional Hardy--Sobolev--Maz'ya inequality does not hold.
\end{abstract}

\maketitle


\section{Introduction}

The fractional Hardy inequalities have garnered significant interest among researchers. Frank and Seiringer \cite{Frank2008} introduced a fractional version of the general Hardy inequalities with the sharp constant by using nonlinear ground state representation. They showed that the following fractional Hardy inequality holds for any function $u \in W^{s,p}(\mathbb{R}^{d})$ when $1 \leq p < \frac{d}{s}$, and for any $u \in W^{s,p}(\mathbb{R}^{d} \setminus \{ 0 \})$ when $p > \frac{d}{s}$, with an optimal constant $C>0$:  
\begin{equation}
\int_{\mathbb{R}^{d}}\int_{\mathbb{R}^{d}} \frac{|u(x)-u(y)|^{p}}{|x-y|^{d+sp}} \, dx \, dy \geq C\int_{\mathbb{R}^{d}} \frac{|u(x)|^{p}}{|x|^{sp}} \, dx.
\end{equation}
Building on this foundation, several generalizations and refinements have been developed, particularly for cases involving singularities at a point (e.g., the origin) and in a half-space, while preserving the optimal constant.  We refer to \cite{Bogdan2011, Frank2009}.  For other (nonweighted) fractional Hardy inequalities, we refer to \cite{AdiJana2024, Bianchi2024, Lieb2008, Loss2010}.

\smallskip

This article focuses on deriving the sharp constant for the weighted fractional Hardy inequalities with singularities on any flat submanifold of codimension $k$, where $1 \leq k < d$. While this inequality was established in \cite{Sahu2025}, the optimality of the sharp constant was not addressed. Here, we fill this gap by establishing the sharp constant. In this article, we also establish a weighted fractional Hardy inequality with a remainder for $p>1$, and using this result, we extend and obtain a weighted version of the fractional Hardy--Sobolev--Maz'ya inequality with singularities on any flat submanifold of codimension $k$, where $1 \leq k < d$. We also take into consideration the case $k=d$ (i.e., singularity at the origin) and establish a suitable weighted fractional Hardy–Sobolev–Maz'ya inequality with a logarithmic weight function.

\smallskip

Weighted fractional Hardy inequalities with singularities at a point (e.g., the origin) and in a half-space were established in \cite{dyda2022sharp} (see also \cite{dyda2024}). Specifically, consider the case of a point singularity. Let $s \in (0,1)$, $p \geq 1$, and $\alpha, \beta, \alpha + \beta \in (-d, sp)$. For all $u \in C_{c}(\mathbb{R}^{d})$ when $sp - \alpha - \beta < d$, and for all $u \in C_{c}(\mathbb{R}^{d} \setminus \{0\})$ when $sp - \alpha - \beta > d$, the following inequality holds with an optimal constant $\mathcal{C}_{1} > 0$:
\begin{equation}\label{Weighted fractional Hardy : point singularity}
\int_{\mathbb{R}^{d}}\int_{\mathbb{R}^{d}} \frac{|u(x)-u(y)|^{p}}{|x-y|^{d+sp}} |x|^{\alpha} |y|^{\beta} \, dx \, dy \geq  \mathcal{C}_{1}  \int_{\mathbb{R}^{d}} \frac{|u(x)|^{p}}{|x|^{sp-\alpha-\beta}} \, dx ,
\end{equation}
where
\begin{equation}\label{The value C 1}
\mathcal{C}_{1} = \mathcal{C}_{1}(d,s,p, \alpha, \beta) = \int_{0}^{1} r^{sp-1} (r^{-\alpha} +r^{-\beta
}) \left| 1-r^{(d+\alpha+\beta-sp)/p} \right|^{p} \Phi_{d,s,p}(r) \, dr.
\end{equation}
Here, 
\begin{equation}
\Phi_{d,s,p}(r) = \begin{dcases}
        |\mathbb{S}^{d-2}| \int_{-1}^{1} \frac{(1-t^{2})^{\frac{d-3}{2}}}{(1-2tr+r^{2})^{\frac{d+sp}{2}}} \, dt, & d\geq 2 \\ 
        (1-r)^{-1-sp}+ (1+r)^{-1-sp}, & d = 1.
    \end{dcases}
\end{equation}

\smallskip

We introduce the weighted fractional Sobolev space with the following definition. Let $s \in (0,1)$, $p \geq 1$, and $\alpha, \beta \in \mathbb{R}$. For any $x \in \mathbb{R}^{d}$, we write $x=(x_{k}, x_{d-k})$, where $1 \leq k < d$, with $k \in \mathbb{N}$, $x_{k} \in \mathbb{R}^{k}$, and $x_{d-k} \in \mathbb{R}^{d-k}$. The weighted Gagliardo seminorm associated with a fixed $k$ is defined as
\begin{equation}\label{weighted Gagliardo}
[u]_{W^{s,p}_{\al,\be;k}(\mathbb{R}^{d})} := \left( \int_{\mathbb{R}^{d}} \int_{\mathbb{R}^{d}}  \frac{|u(x)-u(y)|^{p}}{|x-y|^{d+sp}} |x_{k}|^{\alpha} |y_{k}|^{\beta} \, dx \, dy \right)^{\frac{1}{p}}.
\end{equation}
The corresponding weighted fractional Sobolev space is defined as
\begin{equation}
W^{s,p}_{\al,\be;k}(\mathbb{R}^{d}) := \left\{ u \in L^{p}\left(\mathbb{R}^{d},|x_k|^{\al+\be}\right) : [u]_{W^{s,p}_{\al,\be;k}(\mathbb{R}^{d})} < \infty \right\}.
\end{equation}

For parameters $\alpha, \beta \in (-k, sp)$ satisfying $\alpha + \beta > -k$, the space $C^{1}_{c}(\mathbb{R}^{d})$ is contained in $W^{s,p}_{\al,\be;k}(\mathbb{R}^{d})$ (see \cite[Lemma $2.1$]{Sahu2025}). For convention, when $\alpha=\beta=0$, we denote $ [u]_{W^{s,p}_{0,0;k}(\mathbb{R}^{d})} := [u]_{W^{s,p}(\mathbb{R}^{d})} $ and refer to the space $W^{s,p}_{0,0;k}(\mathbb{R}^{d})= W^{s,p}(\mathbb{R}^{d})$ as the usual fractional Sobolev space. When $k=d$, we omit it in the notation. For further studies on weighted fractional Sobolev spaces, we refer to \cite{valdinoci2015, kijacko2025, kijacko2023}.

\smallskip

For a fixed  $1\leq k<d$, we consider the flat submanifold  
\begin{equation}\label{k defn}
   K:= \{x= (x_{k}, x_{d-k}) \in \mathbb{R}^{k} \times \mathbb{R}^{d-k} : x_{k} =0 \}
\end{equation}
of codimension $k$. For $k=d$, we follow the convention  $K=\{0\}$. In the following theorem, we establish the sharp constant for the weighted fractional Hardy inequality with a singularity on $K$. While a version of this result is presented in \cite[Theorem 1]{Sahu2025}, we provide it here with the explicit sharp constant.


\begin{theorem}[Sharp weighted fractional Hardy inequality]\label{Theorem : Sharp weighted fractional Hardy}
Let $d \geq 1, ~ p\geq 1, ~ s \in (0,1)$ and $1\leq k\leq d, ~ k \in \mathbb{N}$. Assume that $\alpha, \beta \in \mathbb{R}$ are such that $\alpha, \beta, \alpha+\beta \in (-k,sp)$. Let $K$ be a flat submanifold defined in \eqref{k defn}. Then, for all $u \in C^{1}_{c}(\mathbb{R}^{d})$ when $sp<k+ \alpha+\beta$, and for all $u \in C^{1}_{c}(\mathbb{R}^{d} \setminus K)$ when $sp> k+ \alpha+ \beta$, the following inequality holds: 
     \begin{equation}
     \int_{\mathbb{R}^{d}} \int_{\mathbb{R}^{d}} \frac{|u(x)-u(y)|^{p}}{|x-y|^{d+sp}} |x_{k}|^{\alpha} |y_{k}|^{\beta} \, dx \, dy \geq \mathcal{C}   \int_{\mathbb{R}^{d}} \frac{|u(x)|^{p}}{|x_{k}|^{sp-\alpha-\beta}} \, dx , 
    \end{equation}
    where $\mathcal{C}(d,s,p,k,\alpha, \beta)>0$ is the optimal constant given by \eqref{The value C}.
\end{theorem}

Assuming $p \geq 2$, the following theorem provides the sharp weighted fractional Hardy inequality with a remainder term and a singularity on $K$.


\begin{theorem}[Sharp weighted fractional Hardy inequality for $p \geq 2$ with a remainder]\label{Theorem : Sharp weighted fractional Hardy inequality for p geq 2 with a remainder}
Let $d \geq 1, ~ p \geq 2, ~ s \in (0,1)$ and $1\leq k\leq d, ~ k \in \mathbb{N}$. Assume that $\alpha, \beta \in \mathbb{R}$ are such that $\alpha, \beta, \alpha+\beta \in (-k,sp)$. Let $K$ be a flat submanifold defined in \eqref{k defn}. Then, for all $u \in C^{1}_{c}(\mathbb{R}^{d})$ when $sp<k+ \alpha+\beta$, and for all $u \in C^{1}_{c}(\mathbb{R}^{d} \setminus K)$ when $sp> k+ \alpha+ \beta$, the following inequality holds: 
     \begin{equation}
     \begin{split}
     \int_{\mathbb{R}^{d}} \int_{\mathbb{R}^{d}}  \frac{|u(x)-u(y)|^{p}}{|x-y|^{d+sp}} & |x_{k}|^{\alpha} |y_{k}|^{\beta} \, dx \, dy - \mathcal{C}   \int_{\mathbb{R}^{d}} \frac{|u(x)|^{p}}{|x_{k}|^{sp-\alpha-\beta}} \, dx \\ & \geq c_{p} \int_{\mathbb{R}^{d}} \int_{\mathbb{R}^{d}} \frac{|v(x)-v(y)|^{p}}{|x-y|^{d+sp}} |x_{k}|^{-(k-\alpha+\beta-sp)/2} |y_{k}|^{-(k+\alpha-\beta-sp)/2} \, dx \, dy , 
     \end{split}
    \end{equation}
    where $v(x)= |x_{k}|^{(k+\alpha+\beta-sp)/p}u(x)$, $\mathcal{C}(d,s,p,k,\alpha, \beta)>0$ is the optimal constant given by \eqref{The value C}, and $c_{p}$ is given by
    \begin{equation}\label{Value cp}
    c_{p} = \min_{0<\tau<1/2} \left( (1-\tau)^{p} - \tau^{p}+ p \tau^{p-1} \right).
    \end{equation}
    For $p=2$, this inequality becomes an equality with $c_{2} =1$.
\end{theorem}

For the case $p=2$, the unweighted versions of Theorems \ref{Theorem : Sharp weighted fractional Hardy} and \ref{Theorem : Sharp weighted fractional Hardy inequality for p geq 2 with a remainder} were proved by Mallick in \cite{MR3910887}.

\smallskip

Next, we present the sharp weighted fractional Hardy inequality with a remainder for the case $1 < p < 2$. The cases $k = 1$ and $k = d$, corresponding to singularities at a point (the origin) and on a half-space, were established in \cite{dyda2024}. We extend this result to the case of singularities on a flat submanifold $K$ of codimension $k$, where $1 < k < d$. We define $a^{\langle k\rangle} := |a|^k \text{sgn}(a)$. The following theorem states this result.

\begin{theorem}[Sharp weighted fractional Hardy inequality for $1<p< 2$ with a remainder]\label{Theorem : Sharp weighted fractional Hardy inequality for 1<p< 2 with a remainder}
    Let $d \geq 1, ~ 1<p <2, ~ s \in (0,1)$ and $1\leq k\leq d, ~ k \in \mathbb{N}$. Assume that $\alpha, \beta \in \mathbb{R}$ are such that $\alpha, \beta, \alpha+\beta \in (-k,sp)$. Let $K$ be a flat submanifold defined in \eqref{k defn}. Then, for all $u \in C^{1}_{c}(\mathbb{R}^{d})$ when $sp<k+ \alpha+\beta$, and for all $u \in C^{1}_{c}(\mathbb{R}^{d} \setminus K)$ when $sp> k+ \alpha+ \beta$, the following inequality holds: 
     \begin{equation}
     \begin{split}
     \int_{\mathbb{R}^{d}} \int_{\mathbb{R}^{d}}  \frac{|u(x)-u(y)|^{p}}{|x-y|^{d+sp}}  & |x_{k}|^{\alpha} |y_{k}|^{\beta} \, dx \, dy - \mathcal{C}   \int_{\mathbb{R}^{d}} \frac{|u(x)|^{p}}{|x_{k}|^{sp-\alpha-\beta}} \, dx \\ & \geq C_{p} \int_{\mathbb{R}^{d}} \int_{\mathbb{R}^{d}} \frac{\left( v(x)^{\langle p/2 \rangle}-v(y)^{\langle p/2 \rangle} \right)^{2}}{|x-y|^{d+sp}} W(x,y) |x_{k}|^{\alpha} |y_{k}|^{\beta}  \, dx \, dy , 
     \end{split}
    \end{equation}
    where $v(x)= |x_{k}|^{(k+\alpha+\beta-sp)/p}u(x)$, 
    \begin{equation}\label{Defn W(x,y) in theorem}
        W(x,y)= \min \left\{ |x_{k}|^{- \frac{k+\alpha+\beta-sp}{p}}, |y_{k}|^{- \frac{k+\alpha+\beta-sp}{p}} \right\} \max \left\{ |x_{k}|^{- \frac{k+\alpha+\beta-sp}{p}}, |y_{k}|^{- \frac{k+\alpha+\beta-sp}{p}} \right\}^{p-1},
    \end{equation}
     the constant $\mathcal{C}(d,s,p,k,\alpha, \beta)>0$ is  given by \eqref{The value C}, and $C_{p}$ is given by 
    \begin{equation}\label{The value Cp for 1<p<2}
      C_{p}  = \max \left\{ \frac{p-1}{p}, \frac{p(p-1)}{2} \right\}.
    \end{equation}
When $u$ is nonnegative, the constant $C_p$ can be taken as $p-1$.
\end{theorem}

\begin{remark}
    One can show that the statements of Theorems \ref{Theorem : Sharp weighted fractional Hardy}, \ref{Theorem : Sharp weighted fractional Hardy inequality for p geq 2 with a remainder} and \ref{Theorem : Sharp weighted fractional Hardy inequality for 1<p< 2 with a remainder} are true also for $s=0$. The reason is that the weighted Gagliardo seminorm \eqref{weighted Gagliardo} is finite for $s=0$, $\al,\be,\al+\be\in(-k,0)$ and $u\in C_c^1(\R^d)$, in contrast to the unweighted case.
\end{remark}

The following theorem provides the weighted fractional Hardy--Sobolev inequality with a singularity on $K$ for the case $sp \leq d$. This result is instrumental in deriving the weighted fractional Hardy--Sobolev--Maz'ya inequality with a singularity on $K$ for $p >1$.

\begin{theorem}[Weighted fractional Hardy--Sobolev inequality for the case $sp \leq d$]\label{Theorem : Weighted fractional Hardy  sp leq d}
Let $d \geq 2$, $p > 1$, $q > 1$, and $s \in (0,1)$. Suppose that $1 \leq k < d$ with $k \in \mathbb{N}$, and let $\alpha, \beta \in \mathbb{R}$. Define $
\theta = d + (sp - d) \frac{q}{p}$.  
If $sp < d$, assume that $p < q \leq \frac{dp}{d - sp}$, and if $sp = d$, assume that $q > p$. Then, there exists a constant $C = C(d, s, p, q,k, \alpha, \beta) > 0$ such that for all $u \in C^{1}_{c}(\mathbb{R}^{d})$, the following inequality holds:
\begin{equation}\label{Theorem : Weighted fractional Hardy  sp leq d ineq 1}
\begin{split}
      \Bigg( \int_{\mathbb{R}^{d}} \int_{\mathbb{R}^{d}} \frac{|u(x)-u(y)|^{p}}{|x-y|^{d+sp}}  |x_{k}|^{-(k-\alpha+\beta-sp)/2} & |y_{k}|^{-(k+\alpha-\beta-sp)/2} \, dx \,dy  \Bigg)^{\frac{1}{p}} \\ & \geq C \left( \int_{\mathbb{R}^{d}} \frac{|u(x)|^{q}}{|x_{k}|^{\theta + \left( k-sp \right) \frac{q}{p} }} \, dx \right)^{\frac{1}{q}}.
    \end{split}
\end{equation}
Furthermore, define $
\theta' = d + (2s - d) \frac{q'}{2}$.  For any $q'$ satisfying $2<q' \leq \frac{2d}{d-2s}$, and for any $r>1$, we have for any $u \in C^{1}_{c}(\mathbb{R}^{d})$,
\begin{equation}\label{Theorem : Weighted fractional Hardy  sp leq d ineq 2}
  \left(  \int_{\mathbb{R}^{d}} \int_{\mathbb{R}^{d}} \frac{|u(x)-u(y)|^{2}}{|x-y|^{d+2s}} W_{r}(x,y) |x_{k}|^{\alpha} |y_{k}|^{\beta} \, dx \, dy \right)^{\frac{1}{2}} \geq C \left( \int_{\mathbb{R}^{d}} \frac{|u(x)|^{q'}}{|x_{k}|^{\theta' + \left( k-2s \right) \frac{q'}{2} }} \, dx \right)^{\frac{1}{q'}}, 
\end{equation}
where
\begin{equation}\label{Defn: W_r}
    W_{r}(x,y)= \min \left\{ |x_{k}|^{- \frac{k+\alpha+\beta-2s}{r}}, |y_{k}|^{- \frac{k+\alpha+\beta-2s}{r}} \right\} \max \left\{ |x_{k}|^{- \frac{k+\alpha+\beta-2s}{r}}, |y_{k}|^{- \frac{k+\alpha+\beta-2s}{r}} \right\}^{r-1},
\end{equation}
and $C=C(d,s,r,k,\alpha, \beta)$ is a positive constant.
\end{theorem}

By applying Theorem \ref{Theorem : Sharp weighted fractional Hardy inequality for p geq 2 with a remainder} and  the inequality \eqref{Theorem : Weighted fractional Hardy  sp leq d ineq 1} of Theorem \ref{Theorem : Weighted fractional Hardy  sp leq d}, we derive the weighted fractional Hardy--Sobolev--Maz'ya inequality with singularity on a flat submanifold $K$ of codimension $k$, where $1 \leq k < d$ for the case $p \geq 2$. Let $v(x) = |x_{k}|^{(k + \alpha + \beta - sp)/p}u(x)$. Then, we have 
\begin{equation}
\left( \int_{\mathbb{R}^{d}} \frac{|v(x)|^{q}}{|x_{k}|^{\theta + \left( k-sp \right) \frac{q}{p} }} \, dx \right)^{\frac{1}{q}} = \left( \int_{\mathbb{R}^{d}} |u(x)|^{q} |x_{k}|^{\frac{q}{p}(\alpha+\beta) -\theta} \, dx \right)^{\frac{1}{q}}.
\end{equation}
Applying this transformation in  the inequality \eqref{Theorem : Weighted fractional Hardy  sp leq d ineq 1} of Theorem \ref{Theorem : Weighted fractional Hardy  sp leq d} and subsequently using Theorem \ref{Theorem : Sharp weighted fractional Hardy inequality for p geq 2 with a remainder}, we establish the desired weighted fractional Hardy–Sobolev–Maz'ya inequality with singularity on a flat submanifold $K$ defined in \eqref{k defn} for $p \geq 2$. 

\smallskip

For the case $1<p<2$, we apply the inequality \eqref{Theorem : Weighted fractional Hardy  sp leq d ineq 2} of Theorem \ref{Theorem : Weighted fractional Hardy  sp leq d}, making the substitution $ v(x)= |x_{k}|^{(k+\alpha+\beta-sp)/p}u(x)$, replacing $s$ with $\frac{sp}{2}$ (noting that for $1<p<2$, $\frac{sp}{2} \in (0,1)$), and setting $r=p$. This leads to
\begin{equation*}
\begin{split}
    \int_{\mathbb{R}^{d}} \int_{\mathbb{R}^{d}} \frac{\left( v(x)^{\langle p/2 \rangle}-v(y) )^{\langle p/2 \rangle} \right)^{2}}{|x-y|^{d+sp}} W(x,y) &|x_{k}|^{\alpha} |y_{k}|^{\beta}  \, dx \, dy \\ & \geq C   \left( \int_{\mathbb{R}^{d}} |u(x)|^{q} |x_{k}|^{\frac{q}{p}(\alpha+\beta) -\theta} \, dx \right)^{\frac{p}{q}},
\end{split}
\end{equation*}
where  $p < q \leq \frac{dp}{d - sp}$ for $sp<d$. Using Theorem \ref{Theorem : Sharp weighted fractional Hardy inequality for 1<p< 2 with a remainder} along with the above inequality, we derive the weighted fractional Hardy–Sobolev–Maz'ya inequality for the case $1<p<2$. Therefore, we present the following theorem, which establishes weighted fractional Hardy--Sobolev--Maz'ya inequality with a singularity on $K$ for the case $p>1$:


\begin{theorem}[Weighted  fractional Hardy--Sobolev--Maz'ya inequality with singularity on $K$]\label{Theorem : Weighted  fractional Hardy--Sobolev--Maz'ya inequality with singularity on K}
Let $d \geq 2$, $p >1$, $s \in (0,1)$, and $1 \leq k < d$ with $k \in \mathbb{N}$. Suppose that $\alpha, \beta \in \mathbb{R}$ satisfy $\alpha, \beta, \alpha + \beta \in (-k, sp)$, and let $\theta = d + (sp - d)\frac{q}{p}$. Let $K$ be the flat submanifold defined in \eqref{k defn}, and assume $p < q \leq \frac{dp}{d - sp}$ if $sp < d$, and $q > p$ if $sp = d$. Then, for any $u \in C^{1}_{c}(\mathbb{R}^{d})$ when $sp < k + \alpha + \beta$, and for any $u \in C^{1}_{c}(\mathbb{R}^{d} \setminus K)$ when $sp > k + \alpha + \beta$, the following inequality holds:
     \begin{equation}
     \begin{split}
     \int_{\mathbb{R}^{d}} \int_{\mathbb{R}^{d}}  \frac{|u(x)-u(y)|^{p}}{|x-y|^{d+sp}}  |x_{k}|^{\alpha} |y_{k}|^{\beta} \, dx \, dy - \mathcal{C} &  \int_{\mathbb{R}^{d}} \frac{|u(x)|^{p}}{|x_{k}|^{sp-\alpha-\beta}} \, dx  \\ & \geq C_{0} \left( \int_{\mathbb{R}^{d}} |u(x)|^{q} |x_{k}|^{ \frac{q}{p}(\alpha+\beta) -\theta} \, dx \right)^{\frac{p}{q}}, 
     \end{split}
    \end{equation}
where the constant $\mathcal{C}(d,s,p,k,\alpha, \beta)>0$ is  given by \eqref{The value C}, and $C_{0} = C_{0}(d, s, p,q, k, \alpha, \beta)$ is a positive constant.
\end{theorem}

The above theorem also presents a weighted fractional version of the result by Barbatis, Filippas, and Tertikas \cite[Theorem C]{Barbatis2004}. The local (unweighted) analogue is 
\begin{equation}\label{localHSM}
\int_{\R^d}|\nabla u(x)|^p\,dx-\left|\frac{k-p}{p}\right|^p\int_{\R^d}\frac{|u(x)|^p}{|x_k|^p}\,dx\geq C\left(\int_{\R^d}|u(x)|^q|x_k|^{-q-d+\frac{dq}{p}}\,dx\right)^{\frac{p}{q}},
\end{equation}
where $1\leq k<d$, $2\leq p<d$ and $p<q\leq dp/(d-p)$. Interestingly, it is an open question whether the inequality \eqref{localHSM} is true for $1<p<2$, while we know that it is the case in the fractional setting. 

\smallskip

Taking $\alpha = \beta = 0$ in Theorem \ref{Theorem : Weighted  fractional Hardy--Sobolev--Maz'ya inequality with singularity on K}, we get back the unweighted form of the fractional Hardy--Sobolev--Maz'ya inequality, where the singularity lies on a flat submanifold $K$ of codimension $k$, with $1 \leq k < d$.

\begin{remark}\label{Remark}
   When $sp<d$ and $q=\frac{dp}{d-sp}$, we get $\theta=0$. Therefore, in this case, we have the following weighted fractional Hardy--Sobolev--Maz'ya inequality:
  \begin{equation}\label{HSMk<d}
   \begin{split}
     \int_{\mathbb{R}^{d}} \int_{\mathbb{R}^{d}}  \frac{|u(x)-u(y)|^{p}}{|x-y|^{d+sp}}  |x_{k}|^{\alpha} |y_{k}|^{\beta} \, dx \, dy - \mathcal{C} &  \int_{\mathbb{R}^{d}} \frac{|u(x)|^{p}}{|x_{k}|^{sp-\alpha-\beta}} \, dx  \\ & \geq C_{0} \left( \int_{\mathbb{R}^{d}} |u(x)|^{q} |x_{k}|^{ \frac{q}{p}(\alpha+\beta)} \, dx \right)^{\frac{p}{q}}.  
     \end{split}
    \end{equation} 
\end{remark}

\smallskip

A weighted fractional Hardy--Sobolev--Maz'ya inequality in a half-space for $1 < p < 2$ was established in \cite[Theorem 4]{dyda2024}, but with certain conditions on parameters $\al,\be$. In our result, Theorem \ref{Theorem : Weighted fractional Hardy--Sobolev--Maz'ya inequality with singularity on K}, we do not impose any such restrictions on these parameters. In particular, we establish a weighted fractional Hardy--Sobolev--Maz'ya inequality with singularity on $K$ for $p > 1$ without any additional constraints on the parameters.

\smallskip

For the local (unweighted) case when  $k = d$, the authors in \cite[Theorem C]{Barbatis2004} proved a Hardy--Sobolev--Maz'ya inequality with a suitable logarithmic weight function. They also showed that this logarithmic weight function is the best possible choice. Specifically, they proved that for any bounded domain $\Omega$ in $\mathbb{R}^{d}$ with $D> \sup_{\Omega} |x| $, there exists a constant $C>0$ such that for all $ u \in W^{1,p}_{0} (\Omega)$,
\begin{equation*}
    \int_{\Omega}|\nabla u(x)|^p\,dx-\left|\frac{d-p}{p}\right|^p\int_{\Omega}\frac{|u(x)|^p}{|x|^p}\,dx\geq C\left(\int_{\Omega}|u(x)|^q|x|^{-q-d+\frac{dq}{p}} X^{1+\frac{q}{p}} \left( \frac{|x|}{D}  \right) \,dx\right)^{\frac{p}{q}},
\end{equation*}
where $X(t)= -1/ \ln t$ for $t \in (0,1)$, $1<p<d$, and $p \leq q < dp/(d-p)$.

\smallskip

To the best of our knowledge, the weighted (and nonweighted) fractional Hardy--Sobolev--Maz'ya inequality for the case $k=d$ (i.e., singularity at the origin) has not been established in the literature. Furthermore, for $k = d$, using our method of proof, we cannot expect a similar type of weighted fractional Hardy--Sobolev--Maz'ya inequality with singularity on $ K$ of codimension $ k $, where $ 1 \leq k < d $, as given in Theorem \ref{Theorem : Weighted fractional Hardy--Sobolev--Maz'ya inequality with singularity on K}. The reason is that Theorem \ref{Theorem : Weighted fractional Hardy sp leq d} does not hold when  $k = d $, because  $\theta + (d-sp)q/p = d$, and the function \( |u(x)|^q|x|^{-d} \) is not integrable if  $u$  has a nonzero constant value near the origin. Moreover, we establish in Theorem \ref{Theorem : HSM fails k=d} that, unlike the case $1 \leq k < d$ discussed in Remark \ref{Remark}, the Hardy--Sobolev--Maz'ya inequality \eqref{HSMk<d} fails for $k = d$.

\smallskip

Building on the weighted fractional Hardy--Sobolev--Maz'ya inequality with a remainder term for the case $k = d$, as given in \cite[Theorem $1.6$]{dyda2022sharp} for $p \geq 2$ and \cite[Theorem $2$]{dyda2024} for $1 < p < 2$, we establish a weighted fractional Hardy--Sobolev--Maz'ya type inequality with a suitable logarithmic weight function. The following theorem presents our result.
\begin{theorem}[Weighted  logarithmic fractional Hardy--Sobolev--Maz'ya inequality]\label{Theorem 6}
   Let $d \geq 1$, $p >1$ and $s \in (0,1)$. Suppose that $\alpha, \beta \in \mathbb{R}$ satisfy $\alpha, \beta, \alpha + \beta \in (-d, sp)$, and let $\theta = d + (sp - d)\frac{q}{p}$. Assume $p \leq q \leq \frac{dp}{d - sp}$ if $sp < d$, and $q \geq p$ if $sp = d$. Then, for any $u \in C^{1}_{c}(\mathbb{R}^{d})$ when $sp < d + \alpha + \beta$, and for any $u \in C^{1}_{c}(\mathbb{R}^{d} \setminus \{ 0\})$ when $sp > d + \alpha + \beta$ such that $\operatorname{supp} u \subset B(0,R)$ for some $R>0$, the following inequality holds:
     \begin{equation}
     \begin{split}
     \int_{\mathbb{R}^{d}} \int_{\mathbb{R}^{d}}  \frac{|u(x)-u(y)|^{p}}{|x-y|^{d+sp}}  |x|^{\alpha} |y|^{\beta} \, dx \, dy - \mathcal{C}_{1} &  \int_{\mathbb{R}^{d}} \frac{|u(x)|^{p}}{|x|^{sp-\alpha-\beta}} \, dx  \\ & \geq C_{0} \left( \bigintssss_{\mathbb{R}^{d}} \frac{|u(x)|^{q}}{ \ln^{q} \left( \frac{4R}{|x|} \right)} |x|^{ \frac{q}{p}(\alpha+\beta) -\theta} \, dx \right)^{\frac{p}{q}},  
     \end{split}
    \end{equation}
where $\mathcal{C}_{1}(d,s,p,\alpha,\beta)$ is the constant given by \eqref{The value C 1}, and $C_{0} = C_{0}(d, s, p,q, \alpha, \beta)$ is a positive constant. 
\end{theorem}   

\smallskip

This article is structured as follows. Section \ref{Sec : General Hardy inequalities} provides an overview of general Hardy inequalities and the nonlinear ground state representation, which establishes the sharpness of Theorem \ref{Theorem : Sharp weighted fractional Hardy}. This section also includes the proofs of Theorem \ref{Theorem : Sharp weighted fractional Hardy inequality for p geq 2 with a remainder} and Theorem \ref{Theorem : Sharp weighted fractional Hardy inequality for 1<p< 2 with a remainder}. In Section \ref{Sec : Proof of Theorem sp leq d}, we prove Theorem \ref{Theorem : Weighted fractional Hardy sp leq d}, which plays a key role in deriving the weighted fractional Hardy--Sobolev--Maz'ya inequality stated in Theorem \ref{Theorem : Weighted fractional Hardy--Sobolev--Maz'ya inequality with singularity on K}. Section \ref{The case k=d} addresses the case $k=d$, proving Theorem \ref{Theorem 6}, which extends the weighted fractional Hardy--Sobolev--Maz'ya inequality to point singularities (the origin) with an appropriate logarithmic weight function.

\section{General Hardy inequalities}\label{Sec : General Hardy inequalities}
The derivation of the sharp constant in Theorem \ref{Theorem : Sharp weighted fractional Hardy} and the proof of Theorem \ref{Theorem : Sharp weighted fractional Hardy inequality for p geq 2 with a remainder} and Theorem \ref{Theorem : Sharp weighted fractional Hardy inequality for 1<p< 2 with a remainder} rely on the general fractional Hardy inequalities and the nonlinear ground state representation introduced by Frank and Seiringer \cite{Frank2008}. Following the framework in \cite{Frank2008}, consider a symmetric, non-negative, measurable kernel $k(x,y)$ and a non-empty open set $\Omega \subset \mathbb{R}^{d}$. We define the functional
\begin{equation}
E[u] := \int_{\Omega} \int_{\Omega} |u(x)-u(y)|^{p} k(x,y) \, dy \, dx
\end{equation} 
and 
\begin{equation}
V_{\varepsilon} (x) := 2 \omega(x)^{-p+1} \int_{\Omega} (\omega(x)-\omega(y)) |\omega(x)-\omega(y)|^{p-2} k_{\varepsilon}(x,y) \, dy,
\end{equation}
where $\omega$ is a positive, measurable function on $\Omega$ and $\{k_{\varepsilon}(x,y) \}_{\varepsilon>0}$ is a family of measurable, symmetric kernels satisfying the assumptions $0 \leq k_{\varepsilon}(x,y) \leq k(x,y)$, $\lim_{\varepsilon \to 0} k_{\varepsilon}(x,y) = k(x,y)$ for all $x,y \in \Omega$. Assuming the integrals that define $V_{\varepsilon}(x)$ are absolutely convergent for almost every $x \in \Omega$, as $\ve\rightarrow 0^+$, and that $V_{\varepsilon}$ converges weakly in $L^{1}_{\text{loc}}(\Omega)$ to some function $V$, we obtain the following Hardy-type inequality:
\begin{equation}\label{General Hardy inequality}
E[u] \geq \int_{\Omega} |u(x)|^{p}V(x) \, dx 
\end{equation}
for any compactly supported $u$ satisfying $\int_{\Omega} |u(x)|^{p} V_{+}(x) \, dx < \infty$ (see \cite[Proposition $2.2$]{Frank2008}).

\smallskip

Furthermore, if $p \geq 2$, the inequality above can be refined by incorporating a remainder term defined by
\begin{equation}
E_{\omega}[v] := \int_{\Omega} \int_{\Omega} |v(x)-v(y)|^{p} \omega(x)^{p/2} \omega(y)^{p/2} k(x,y) \, dy \, dx, \hspace{3mm} u=\omega v.
\end{equation}
Specifically, we have
\begin{equation}\label{General Hardy with weight}
E[u] - \int_{\Omega} |u(x)|^{p}V(x) \, dx \geq c_{p} E_{\omega}[v] 
\end{equation} 
under the same conditions as above, where $c_{p}$ is given by \eqref{Value cp}. When $p=2$, the inequality \eqref{General Hardy with weight} becomes an equality.

\smallskip

For the case $1<p<2$, the remainder term was found in \cite{dyda2024} (see also \cite{MR4597627}). Denoting
\begin{equation*} \widetilde{E}_{\omega}[u]:=\int_{\Omega}\int_{\Omega}\left(u(x)^{\langle p/2 \rangle }-u(y)^{\langle p/2 \rangle }\right)^2W(x,y)k(x,y)\,dy\,dx,
\end{equation*}
where 
\begin{equation*}
    W(x,y):=\min\{\omega(x),\omega(y)\}\max\{\omega(x),\omega(y)\}^{p-1}=\omega(x)\omega(y)\max\{\omega(x),\omega(y)\}^{p-2},
\end{equation*}
and
$a^{\langle t \rangle}:=|a|^{t}\text{sgn}(a)$ is the so-called \emph{French power}, we have 
\begin{equation}\label{General Hardy with weight1<p<2}
E[u] - \int_{\Omega} |u(x)|^{p}V(x) \, dx \geq C_{p} \widetilde{E}_{\omega}[v] , \hspace{3mm} u=\omega v,
\end{equation} 
where $C_{p}$ is given by \eqref{The value Cp for 1<p<2}. When $u$ is a nonnegative function, $C_p$ can be taken as $p-1$.

\smallskip

\subsection{Proof of Theorem \ref{Theorem : Sharp weighted fractional Hardy}}
We will now establish the sharpness of the constant in Theorem \ref{Theorem : Sharp weighted fractional Hardy}. First, we prove Theorem \ref{Theorem : Sharp weighted fractional Hardy} using a nonlinear ground state representation and general fractional Hardy inequalities. Then, we illustrate that the obtained constant is sharp.

\smallskip

For any $x \in \mathbb{R}^{d}$, we write $x = (x_{k}, x_{d-k})$, where $x_{k} \in \mathbb{R}^{k}$ and $x_{d-k} \in \mathbb{R}^{d-k}$. Following the ideas in \cite{dyda2024, dyda2022sharp, Frank2008}, we introduce the notation
\begin{equation*}
\gamma=\frac{k+\alpha+\beta-sp}{p}, \hspace{4mm} \omega(x) = |x_{k}|^{- \gamma},
\end{equation*}
\begin{equation*}
k(x,y) = \frac{1}{2} |x-y|^{-d-sp} \left(|x_{k}|^{\alpha}|y_{k}|^{\beta
} + |x_{k}|^{\beta}|y_{k}|^{\alpha}\right).
\end{equation*}
Furthermore, we set
\begin{equation}\label{The value C}
\mathcal{C}(d,s,p,k, \alpha, \beta) =  \frac{\pi^{(d-k)/2}\Gamma \left( \frac{k+sp}{2}  \right)}{\Gamma \left( \frac{d+sp}{2} \right)} \mathcal{C}_{1}(k,s,p, \alpha, \beta) ,
\end{equation}
where $\mathcal{C}_{1}(k,s,p,\alpha,\beta)$ is defined in \eqref{The value C 1}, and $\Gamma$ denotes the Gamma function. We now present the following lemma.
\begin{lemma}\label{Lemma : on flat submanifold}
Let $\alpha,  \beta,  \alpha+ \beta \in (-k,sp)$. Then
\begin{equation}
2 \lim_{\varepsilon \to 0} \int_{||x_{k}|-|y_{k}||> \varepsilon} (\omega(x) - \omega(y))|\omega(x)-\omega(y)|^{p-2}k(x,y) \, dy = \frac{\mathcal{C}(d,s,p,k, \alpha,\beta)}{|x_{k}|^{sp-\alpha-\beta}} \omega(x)^{p-1}
\end{equation}
uniformly on compact sets contained in $\mathbb{R}^{d} \backslash K$. The constant $\mathcal{C}(d,s,p,k, \alpha,\beta)$ is given by \eqref{The value C}.
\end{lemma}
\begin{proof}
We start by observing that
\begin{equation*}
 \int_{||x_{k}|-|y_{k}||> \varepsilon} (\omega(x) - \omega(y))|\omega(x)-\omega(y)|^{p-2}k(x,y) \, dy = |x_{k}|^{\alpha}I_{\varepsilon}(\beta) + |x_{k}|^{\beta} I_{\varepsilon}(\alpha),
\end{equation*}
where
\begin{equation*}
\begin{split}
    I_{\varepsilon}(a) & = \int_{||x_{k}|-|y_{k}||> \varepsilon}  \frac{(|x_{k}|^{-\gamma}-|y_{k}|^{- \gamma}) \left|  |x_{k}|^{-\gamma}-|y_{k}|^{-\gamma} \right|^{p-2} |y_{k}|^{a}}{|x-y|^{d+sp}} \, dy \\ & = \int_{||x_{k}|-|y_{k}||> \varepsilon}  \frac{(|x_{k}|^{-\gamma}-|y_{k}|^{- \gamma}) \left|  |x_{k}|^{-\gamma}-|y_{k}|^{-\gamma} \right|^{p-2} |y_{k}|^{a}}{ \left(|x_{k}-y_{k}|^{2} + |x_{d-k}-y_{d-k}|^{2} \right)^{\frac{d+sp}{2}}} \, dy .
\end{split}
\end{equation*}
By applying the substitution $y_{d-k} \to x_{d-k} + y_{d-k}|x_{k}-y_{k}|$, we obtain
\begin{equation*}
I_{\varepsilon}(a) = \int_{\mathbb{R}^{d-k}} \frac{1}{(1+|y_{d-k}|^{2})^{\frac{d+sp}{2}}}  \, dy_{d-k} \times J_{\varepsilon}(a),
\end{equation*}
where
\begin{equation*}
J_{\varepsilon}(a)= \int_{||x_{k}|-|y_{k}||> \varepsilon}  \frac{(|x_{k}|^{-\gamma}-|y_{k}|^{- \gamma}) \left|  |x_{k}|^{-\gamma}-|y_{k}|^{-\gamma} \right|^{p-2} |y_{k}|^{a}}{|x_{k}-y_{k}|^{k+sp}} \, dy_{k}.
\end{equation*}
Furthermore, let  $\mathbb{S}^{d-k-1}$ be the unit sphere in $\mathbb{R}^{d-k}$. Its measure is given by $|\mathbb{S}^{d-k-1}| = \frac{2 \pi^{(d-k)/2}}{\Gamma \left( \frac{d-k}{2}\right)}$. Using \cite[formula (6.2.1)]{Abramowitz1992}, we obtain
\begin{equation*}
    \int_{\mathbb{R}^{d-k}} \frac{1}{(1+|y_{d-k}|^{2})^{\frac{d+sp}{2}}} \, dy_{d-k} = \frac{\pi^{(d-k)/2} \Gamma \left( \frac{k+sp}{2} \right)}{\Gamma \left( \frac{d+sp}{2} \right)}.
\end{equation*}
Finally, by applying \cite[Lemma 3.1]{dyda2022sharp} in $\mathbb{R}^{k}$, we complete the proof.
\end{proof}

\smallskip

\begin{proof}[Proof of Theorem \ref{Theorem : Sharp weighted fractional Hardy}]
By following the approach in \cite{dyda2022sharp, Frank2008}, we apply the substitution $u(x) = v(x) \omega(x)$ and utilize the inequality \eqref{General Hardy inequality} along with Lemma \ref{Lemma : on flat submanifold}. This allows us to establish that for any $u \in C^{1}_{c}(\mathbb{R}^{d})$ if $sp<k+ \alpha+\beta$, and $u \in C^{1}_{c}(\mathbb{R}^{d} \backslash K)$ if $sp>k+\alpha+\beta$, with the conditions $\alpha, \beta, \alpha+\beta \in (-k,sp)$, the following inequality holds:
\begin{equation}
     \int_{\mathbb{R}^{d}} \int_{\mathbb{R}^{d}} \frac{|u(x)-u(y)|^{p}}{|x-y|^{d+sp}} |x_{k}|^{\alpha} |y_{k}|^{\beta} \, dx \, dy \geq \mathcal{C}   \int_{\mathbb{R}^{d}} \frac{|u(x)|^{p}}{|x_{k}|^{sp-\alpha-\beta}} \, dx,
    \end{equation}
    where $\mathcal{C}$ is given by \eqref{The value C}. Next, we will illustrate the optimality of the constant $\mathcal{C}$. Let
\begin{equation}
\mathfrak{h}^{k,d}_{s,p,\al,\be} := \inf_{u \in X,\,u\neq 0} \frac{\displaystyle\int_{\mathbb{R}^{d}} \int_{\mathbb{R}^{d}} \frac{|u(x)-u(y)|^{p}}{|x-y|^{d+sp}} |x_{k}|^{\alpha} |y_{k}|^{\beta} \, dx \, dy}{\displaystyle\int_{\mathbb{R}^{d}} \frac{|u(x)|^{p}}{|x_{k}|^{sp-\alpha-\beta}} \, dx},
\end{equation}
where $X= C^{1}_{c}(\mathbb{R}^{d})$ if $sp< k+ \alpha+ \beta$, and $X= C^{1}_{c}(\mathbb{R}^{d} \setminus K)$ if $sp> k+ \alpha+ \beta$. Then, we have
\begin{equation}\label{Sharpness 1}
\mathfrak{h}^{k,d}_{s,p,\al,\be} \geq \mathcal{C}(d,s,p,k, \alpha,\beta),
\end{equation}
where $\mathcal{C}(d,s,p,k, \alpha,\beta)$ is defined in \eqref{The value C}. Let $\eta \in C^{\infty}_{c}(\mathbb{R}^{k})$ when $sp-\alpha-\beta<k$, and $\eta \in C^{\infty}_{c}(\mathbb{R}^{k}\setminus K)$ (note that $C^{\infty}_{c}(\mathbb{R}^{k}\setminus K) = C^{\infty}_{c}(\mathbb{R}^{k}\setminus \{ 0 \})$) when $sp-\alpha-\beta>k$, and let $\phi \in C^{\infty}_{c}(\mathbb{R}^{d-k})$. 
For $N>0$, let
\begin{equation*}
\phi_{N}(x_{d-k}) = \frac{N^{\frac{k-d}{p}}}{\| \phi \|_{L^{p}(\mathbb{R}^{d-k})}} \phi \left( \frac{x_{d-k}}{N}  \right).
\end{equation*}
We define the function $u_{N}(x) = \eta(x_{k}) \phi_{N}(x_{d-k}) $. Then, we have
\begin{equation}\label{Sharp ineq 1}
\mathfrak{h}^{k,d}_{s,p,\al,\be} \leq \frac{[\eta \phi_{N}]^{p}_{W^{s,p
}_{ \alpha, \beta;k}(\mathbb{R}^{d})}}{\displaystyle\int_{\mathbb{R}^{d}} \frac{| \eta(x_{k}) \phi_{N}(x_{d-k}) |^{p}}{|x_{k}|^{sp-\alpha-\beta}} dx}  = \frac{[\eta \phi_{N}]^{p}_{{W^{s,p
}_{ \alpha, \beta;k}}(\mathbb{R}^{d})}}{\displaystyle\int_{\mathbb{R}^{k}} \frac{| \eta(x_{k}) |^{p}}{|x_{k}|^{sp-\alpha-\beta}} \, dx_{k}} . 
\end{equation}
Using Minkowski inequality, we get
\begin{equation}\label{Sharp ineq 2}
\begin{split}
[\eta \phi_{N}]_{{W^{s,p
}_{ \alpha, \beta;k}}(\mathbb{R}^{d})} & \leq \left( \int_{\mathbb{R}^{d}} \int_{\mathbb{R}^{d}} \frac{|\phi_{N}(x_{d-k})|^{p} |\eta(x_{k})-\eta(y_{k})|^{p}}{|x-y|^{d+sp}} |x_{k}|^{\alpha} |y_{k}|^{\beta} \, dx \, dy \right)^{\frac{1}{p}} \\ & \hspace{5mm} + \left( \int_{\mathbb{R}^{d}} \int_{\mathbb{R}^{d}} \frac{|\eta(x_{k})|^{p} |\phi_{N}(x_{d-k})-\phi_{N}(y_{d-k})|^{p}}{|x-y|^{d+sp}} |x_{k}|^{\alpha} |y_{k}|^{\beta} \, dx \, dy \right)^{\frac{1}{p}} \\ & =:I_{1}+I_{2}.
\end{split}
\end{equation}
For the integral $I_{1}$, using the substitution $y_{d-k} \to x_{d-k} + y_{d-k}|x_{k}-y_{k}|$, we get
\begin{equation*}
\begin{split}
I^{p}_{1} & =\int_{\mathbb{R}^{d}} \int_{\mathbb{R}^{d}} \frac{|\phi_{N}(x_{d-k})|^{p} |\eta(x_{k})-\eta(y_{k})|^{p}}{|x-y|^{d+sp}} |x_{k}|^{\alpha} |y_{k}|^{\beta} \, dx \, dy \\ & =  \int_{\mathbb{R}^{d-k}} \frac{1}{(1+|y_{d-k}|^{2})^{\frac{d+sp}{2}}} dy_{d-k} \times [\eta]^{p}_{W^{s,p}_{ \alpha, \beta}(\mathbb{R}^{k})} \\ & = \frac{\pi^{(d-k)/2} \Gamma \left( \frac{k+sp}{2} \right)}{\Gamma \left( \frac{d+sp}{2} \right)} \times [\eta]^{p}_{W^{s,p}_{ \alpha, \beta}(\mathbb{R}^{k})}.  
\end{split}
\end{equation*}
For $I_{2}$, we have
\begin{equation*}
\begin{split}
I^{p}_{2} & =  \int_{\mathbb{R}^{d}} \int_{\mathbb{R}^{d}} \frac{|\eta(x_{k})|^{p} |\phi_{N}(x_{d-k})-\phi_{N}(y_{d-k})|^{p}}{|x-y|^{d+sp}} |x_{k}|^{\alpha} |y_{k}|^{\beta} \, dx \, dy\\
&=\int_{\R^{k}}|\eta(x_k)|^p|x_k|^{\al}dx_k\int_{\R^{d-k}}\int_{\R^{d-k}}|\phi_{N}(x_{d-k})-\phi_{N}(y_{d-k})|^{p}|\,dx_{d-k}\,dy_{d-k} \\ & \hspace{5mm} \times \int_{\R^k}\frac{|y_k|^{\be}}{|x-y|^{d+sp}}\,dy_k.
\end{split} 
\end{equation*}
To show that $I^{p}_{2} \to 0$ as $N \to \infty$, we proceed as follows. Let
$$
J:=\int_{\R^k}\frac{|y_k|^{\be}}{|x-y|^{d+sp}}\,dy_k.
$$
Let us first consider the case when $\beta\in (-k,0]$. By Young's inequality, 
\begin{align*}
|x-y|^{d+sp}&=\left(|x_k-y_k|^2+|x_{d-k}-y_{d-k}|^2\right)^{\frac{d+sp}{2}}
\geq C_1|x_k-y_k|^{k-\varepsilon}|x_{d-k}-y_{d-k}|^{d-k+sp+\varepsilon},
\end{align*}
and similarly,
\begin{align*}
|x-y|^{d+sp}&\geq C_2|x_k-y_k|^{k+\varepsilon}|x_{d-k}-y_{d-k}|^{d-k+sp-\varepsilon},
\end{align*}
where $\varepsilon>0$ is sufficiently small. Hence, using the fact that the function $w(x)=|x_k|^\be$ belongs to the Muckenhoupt class $A_1$ \cite[Theorem 1.1]{MR3900847} and the maximal function estimate from \cite{MR312232}, we have for some positive constant $C(d,k,s,p,\varepsilon,\be)$ (which may vary from line to line) that
\begin{align*}
J &\leq C\int_{|x_k-y_k|<1}\frac{|y_k|^{\be}}{|x_k-y_k|^{k-\ve}|x_{d-k}-y_{d-k}|^{d-k+sp+\ve}}\,dy_k\\
&\quad+\int_{|x_k-y_k|>1}\frac{|y_k|^{\be}}{|x_k-y_k|^{k+\ve}|x_{d-k}-y_{d-k}|^{d-k+sp-\ve}}\,dy_k\\
&\leq C|x_k|^{\be}\left(|x_{d-k}-y_{d-k}|^{-(d-k+sp+\ve)}+|x_{d-k}-y_{d-k}|^{-(d-k+sp-\ve)}\right).
\end{align*}
If $\ve$ is sufficiently small, we have $sp-\ve=s_1p$ and $sp+\ve=s_2p$ for some $s_1,\,s_2\in(0,1)$ (these two conditions define $\ve$). Therefore, we obtain
\begin{align*}
I_2^p&\leq C\int_{\R^k}|\eta(x_k)|^p|x_k|^{\al+\be}\,dx_k\left([\phi_N]^p_{W^{s_1,p}(\R^{d-k})}+[\phi_N]^p_{W^{s_2,p}(\R^{d-k})}\right)\\
&\quad=\frac{\displaystyle\int_{\R^k}|\eta(x_k)|^p|x_k|^{\al+\be}\,dx_k}{\|\phi\|^p_{L^p(\R^{d-k})}}\left(\frac{[\phi]^p_{W^{s_1,p}(\R^{d-k})}}{N^{s_1p}}+\frac{[\phi]^p_{W^{s_2,p}(\R^{d-k})}}{N^{s_2p}}\right)\rightarrow0,
\end{align*}
when $N\rightarrow\infty$.

We now turn to the case $\beta\in(0,sp)$. Again by Young's inequality for sufficiently small $\varepsilon>0$,
\begin{align*}
J&\leq\int_{|y_k|\leq|x_k|}\frac{|y_k|^\be}{|x_k-y_k|^{k-\ve}|x_{d-k}-y_{d-k}|^{d-k+sp+\ve}}\,dy_k\\
&\quad+\int_{|y_k|\geq|x_k|}\frac{|y_k|^{\be}}{|x_k-y_k|^{k+sp-\ve}|x_{d-k}-y_{d-k}|^{d-k+\ve}}\,dy_k\\
&\leq C |x_k|^{\be+\ve}|x_{d-k}-y_{d-k}|^{-(d-k+sp+\ve)}+C|x_k|^{\be-sp+\ve}|x_{d-k}-y_{d-k}|^{-(d-k+\ve)}.
\end{align*}
The above estimate leads to
\begin{align*}
 I_2^p&\leq C \frac{1}{N^{s_2p}}\frac{\displaystyle\int_{\R^k}|\eta(x_k)|^p|x_k|^{\al+\be+\ve}dx_k}{\|\phi\|_{L^p(\R^{d-k})}^p}[\phi]^p_{W^{s_2,p}(\R^{d-k})}\\
 &\quad+\frac{1}{N^{\ve}}\frac{\displaystyle\int_{\R^k}|\eta(x_k)|^p|x_k|^{\al+\be-sp+\ve}dx_k}{\|\phi\|_{L^p(\R^{d-k})}^p} [\phi]^p_{W^{\ve/p,p}(\R^{d-k})}\rightarrow 0,
\end{align*}
as $N\rightarrow\infty$. Notice that in both cases we used the fact that $C^{\infty}_c(\R^{d-k})\subset W^{\sigma,p}(\R^{d-k})$ for any $\sigma\in(0,1)$.

By combining $I^{p}_{1}$ and $I^{p}_{2}$ in \eqref{Sharp ineq 2}, utilizing \eqref{Sharp ineq 1}, and letting $N \to \infty$, we deduce that
\begin{equation*}
    \mathfrak{h}^{k,d}_{s,p,\al,\be} \leq \frac{\pi^{(d-k)/2} \Gamma \left( \frac{k+sp}{2} \right)}{\Gamma \left( \frac{d+sp}{2} \right)} \times \frac{[\eta ]^{p}_{W^{s,p}_{ \alpha, \beta}
}(\mathbb{R}^{k})}{\displaystyle\int_{\mathbb{R}^{k}} \frac{| \eta(x_{k}) |^{p}}{|x_{k}|^{sp-\alpha-\beta}} dx_{k}} .
\end{equation*}
Applying the sharpness of $\mathcal{C}_{1}$ in the inequality \eqref{Weighted fractional Hardy : point singularity}, along with the definition of $\mathcal{C}$, we get
\begin{equation*}
   \mathfrak{h}^{k,d}_{s,p,\al,\be} \leq \mathcal{C}(d,s,p,k, \alpha,\beta).
\end{equation*}
This proves the sharpness of the constant in Theorem \ref{Theorem : Sharp weighted fractional Hardy}.
\end{proof}

\subsection{Proof of Theorem \ref{Theorem : Sharp weighted fractional Hardy inequality for p geq 2 with a remainder} and Theorem \ref{Theorem : Sharp weighted fractional Hardy inequality for 1<p< 2 with a remainder}}
Theorem \ref{Theorem : Sharp weighted fractional Hardy inequality for p geq 2 with a remainder} follows directly from Lemma \ref{Lemma : on flat submanifold} and the general fractional Hardy inequality with a remainder for the case $p \geq 2$ (see \eqref{General Hardy with weight}). Similarly, the proof of Theorem \ref{Theorem : Sharp weighted fractional Hardy inequality for 1<p< 2 with a remainder} is obtained using Lemma \ref{Lemma : on flat submanifold} and the general fractional Hardy inequality with a remainder for the case $1<p<2$  (see \eqref{General Hardy with weight1<p<2}).

\section{Proof of Theorem \ref{Theorem : Weighted fractional Hardy  sp leq d}}\label{Sec : Proof of Theorem sp leq d}

In this section, we present the proof of Theorem \ref{Theorem : Weighted fractional Hardy  sp leq d}, where we establish a weighted fractional Hardy--Sobolev inequality with singularities on a flat submanifold of codimension $k$, for $1\leq k<d$ in the case $sp \leq d$. Building upon the ideas from \cite{Sahu2025} with suitable modifications, we prove Theorem \ref{Theorem : Weighted fractional Hardy  sp leq d}.

\smallskip

For a measurable set $\Omega \subset \mathbb{R}^d, ~ (u)_{\Omega}$ will denote the average of the function $u$ over $\Omega$, i.e., 
\begin{equation*}
    (u)_{\Omega} :=  \frac{1}{|\Omega|} \int_{\Omega} u(x) \, dx = \fint_{\Omega} u(x) \, dx. 
\end{equation*}
Here, $|\Omega|$ represents the Lebesgue measure of $\Omega$. For any ~$a_{1},  \dots , a_{m} \in \mathbb{R}$ and ~$\gamma \geq1$, the following inequality holds:
\begin{equation}\label{sumineq}
    \sum_{\ell=1}^{m} |a_{\ell}|^{\gamma} \leq  \left( \sum_{\ell=1}^{m} |a_{\ell}| \right)^{\gamma} .
\end{equation}

\smallskip

The following lemma plays a crucial role in proving Theorem \ref{Theorem : Sharp weighted fractional Hardy inequality for 1<p< 2 with a remainder}. It provides a fractional Sobolev inequality for the case $sp \leq d$, when bounded Lipschitz domains are dilated by a factor $\lambda > 0$. The proof of this result can be found in \cite[Lemma 2.1]{Adi2024} (also in \cite{Adi2023}).

\begin{lemma}\label{sobolev} 
Let $\Omega$ be a bounded Lipschitz domain in $\mathbb{R}^{d}$. Suppose $p > 1$ and $s \in (0,1)$ satisfy $sp \leq d$. For $\lambda > 0$, define the rescaled domain $\Omega_{\lambda}: = \left\{ \lambda x : ~ x \in \Omega \right\}$. Then, there exists a constant $C = C(d,s,p, q, \Omega) > 0$ such that for any function $u \in W^{s, p}(\Omega_{\lambda})$, the following inequality holds:
    \begin{equation}
         \left( \fint_{\Omega_{\lambda}} |u(x)-(u)_{\Omega_{\lambda}}|^{q } dx \right)^{\frac{1}{q}}  \leq C \left( \lambda^{sp-d} [u]^{p}_{W^{s, p}(\Omega_{\lambda})} \right)^{\frac{1}{p}}  ,
    \end{equation}
    where $q \in [p, p^{*}_{s}]$ if $sp<d$, and $q \geq p$ if $sp=d$.
\end{lemma} 

\smallskip

We now state a fundamental inequality, the proof of which is available in \cite[Lemma 2.5]{Adi2024} (also in \cite{Adi2023}).

\begin{lemma}\label{Lemma: estimate}
    Let $q >1$ and $c>1$. Then for all $a, b \in \mathbb{R}$, we have 
    \begin{equation}
        (|a| + |b|)^{q} \leq c|a|^{q} + (1-c^{-\frac{1}{q -1}})^{1-q} |b|^{q} .
    \end{equation}
\end{lemma}

\bigskip

Fix $1 \leq k < d, ~ k \in \mathbb{N}$. Let $ \mathcal{D}:= B^{k}(0, 2^{n_{0}+1}) \times (-2^{n_{1}},2^{n_{1}})^{d-k}$, where $n_{0}, n_{1} \in \mathbb{Z}$ and $2^{n_{1}} \geq 2^{n_{0}+1}$. Here, $B^{k}(0, 2^{n_{0}+1})$ denotes a ball centered at $0$ with radius $ 2^{n_{0}+1}$ in $\mathbb{R}^{k}$. Let $x = (x_{k}, x_{d-k}) \in \mathbb{R}^{k} \times \mathbb{R}^{d-k}$, where $x_{k} \in \mathbb{R}^k$ and $x_{d-k} \in \mathbb{R}^{d-k}$. Let $\theta = d+ \left(sp-d \right)\frac{ q }{p}$, where we fix $q > p$ when $ sp=d$, and $q \in \left(p, \frac{dp}{d-sp} \right] $ when $sp<d$. For each $\ell \leq n_{0}$, we define
\begin{equation*}
    \mathcal{B}_{\ell} = \{ x=(x_{k},x_{d-k}) \in \mathcal{D} :  2^{ \ell}  \leq |x_{k}| < 2^{\ell +1}  \}.
\end{equation*}
Therefore, we have
\begin{equation*}
    \mathcal{D} = \bigcup_{\ell= - \infty}^{n_{0}} \mathcal{B}_{\ell}.
\end{equation*}
Again, we further divide $\mathcal{B}_{\ell}$ into a disjoint union of sets of the same size, denoted as $\mathcal{B}^{i}_{\ell}$, such that if $x = (x_{k}, x_{d-k}) \in \mathcal{B}^{i}_{\ell}$, then $x_{d-k} \in C^{i}_{\ell}$, where $C^{i}_{\ell}$  is a cube of side length $2^{ \ell}$ in $\mathbb{R}^{d-k}$. Then, we have
\begin{equation*}
      \mathcal{B}_{\ell} = \bigcup_{i = 1}^{\sigma_{\ell}} \mathcal{B}^{i}_{\ell}  ,
\end{equation*}
where $\sigma_{\ell} = 2^{(- \ell+1)(d-k)} 2^{n_{1}(d-k)}$. The Lebesgue measure of $\mathcal{B}^{i}_{\ell}$ is then given by
\begin{equation*}
    |\mathcal{B}^{i}_{\ell}| = \left( \frac{|\mathbb{S}^{k-1}| (2^{k}-1)  2^{\ell k}}{k} \right) \times 2^{\ell (d-k)} =  \frac{|\mathbb{S}^{k-1}| (2^{k}-1)  2^{\ell d}}{k} .
\end{equation*}
Here, $\mathbb{S}^{k-1}$ represents the boundary of an open ball of radius $1$ with center $0$ in $\mathbb{R}^{k}$.

\smallskip

\begin{proof}[\textbf{Proof of Theorem \ref{Theorem : Weighted fractional Hardy  sp leq d}}]
Let $u \in C^{1}_{c}(\mathbb{R}^{d})$ be such that
\begin{equation*}
\operatorname{supp} \,  u \subset B^{k}(0, 2^{n_{0}+1}) \times (-2^{n_{1}},2^{n_{1}})^{d-k}= \mathcal{D}
\end{equation*}
for some $n_{0}, n_{1} \in \mathbb{Z}$ satisfying $2^{n_{1}} \geq 2^{n_{0}+1}$. Fix any  $\mathcal{B}^{i}_{\ell}$. By applying Lemma \ref{sobolev} with the domain $\Omega = \{ (x_{k}, x_{d-k}) : 1 < |x_{k}| < 2  \ \text{and} \ x_{d-k} \in (1,2)^{d-k} \}$, $ \lambda = 2^{\ell}$, and using translation invariance, we have
\begin{equation}\label{eqn1}
 \fint_{\mathcal{B}^{i}_{\ell}} |u(x)-(u)_{\mathcal{B}^{i}_{\ell}}|^{q}  dx \leq C 2^{\ell \left(sp-d \right) \frac{q}{p} }[u]^{q}_{W^{s, p}(\mathcal{B}^{i}_{\ell})}  ,
\end{equation}
where $C= C(s,p,q)$ is a positive constant. For $x = (x_{k},x_{d-k}) \in \mathcal{B}^{i}_{\ell}$, we have $$C2^{\ell \left( \theta + \left( k-sp \right) \frac{q}{p} \right)} \leq |x_{k}|^{\theta + \left( k-sp \right) \frac{q}{p}}$$ for some positive constant $C=C(d,s,p,q,k,\alpha, \beta)$. Therefore, we have
\begin{equation*}
\begin{split}
    \int_{\mathcal{B}^{i}_{\ell}} \frac{|u(x)|^{q}}{|x_{k}|^{\theta + \left( k-sp \right) \frac{q}{p} }}  dx &\leq \frac{C}{2^{\ell \left(\theta + \left( k-sp \right) \frac{q}{p}\right)}} \int_{\mathcal{B}^{i}_{\ell}} |u(x)-(u)_{\mathcal{B}^{i}_{\ell}} + (u)_{\mathcal{B}^{i}_{\ell}}|^{q}  dx \\
    &\leq \frac{C}{2^{\ell \left( \theta + \left( k-sp \right) \frac{q}{p} \right)}} \int_{\mathcal{B}^{i}_{\ell}} |u(x)-(u)_{\mathcal{B}^{i}_{\ell}}|^{q}  dx + \frac{C}{2^{\ell \left( \theta + \left( k-sp \right) \frac{q}{p} \right)}} \int_{\mathcal{B}^{i}_{\ell}} |(u)_{\mathcal{B}^{i}_{\ell}}|^{q}dx.
\end{split}
\end{equation*}
  By applying inequality \eqref{eqn1} along with the observation that for any $\alpha, \beta \in \mathbb{R}$,
\begin{equation*}
2^{\ell \left( -k+sp \right)}= 2^{\ell (-k+\alpha-\beta + sp)/2 } 2^{\ell (-k-\alpha+\beta + sp)/2 } \leq C |x_{k}|^{-(k-\alpha+\beta-sp)/2} |y_{k}|^{-(k+\alpha-\beta-sp)/2}
\end{equation*}  
for some positive constant $C=C(s,p,k,\alpha,\beta)$, for all $(x_{k}, x_{d-k}), (y_{k}, y_{d-k}) \in \mathcal{B}^{i}_{\ell}$, and utilizing the definition of $\theta$, we obtain
    \begin{equation*}
    \begin{split}
        \int_{\mathcal{B}^{i}_{\ell}} \frac{|u(x)|^{q}}{|x_{k}|^{\theta + \left( k-sp \right) \frac{q}{p}}}  dx &\leq C \frac{|\mathcal{B}^{i}_{\ell}|}{2^{\ell \left(\theta + \left( k-sp \right) \frac{q}{p}\right)}} \fint_{\mathcal{B}^{i}_{\ell}} |u(x)-(u)_{\mathcal{B}^{i}_{\ell}}|^{q}  dx  + C \frac{|\mathcal{B}^{i}_{\ell}|}{2^{\ell \left(\theta + \left( k-sp \right) \frac{q}{p}\right)}} |(u)_{\mathcal{B}^{i}_{\ell}}|^{q} \\
   & \leq  C 2^{\ell \left( -k+sp \right) \frac{q}{p}} [u]^{q}_{W^{s, p}(\mathcal{B}^{i}_{\ell})}  + C 2^{\ell \left( d-k \right) \frac{q}{p}} |(u)_{\mathcal{B}^{i}_{\ell}}|^{q} \\ &\leq  C \left( \int_{\mathcal{B}^{i}_{\ell}} \int_{\mathcal{B}^{i}_{\ell}} \frac{|u(x)-u(y)|^{p}}{|x-y|^{d+sp}} |x_{k}|^{-(k-\alpha+\beta-sp)/2} |y_{k}|^{-(k+\alpha-\beta-sp)/2} \, dx \, dy \right)^{\frac{q}{p}} \\ & \hspace{5mm} + C 2^{\ell \left( d-k \right) \frac{q}{p}} |(u)_{\mathcal{B}^{i}_{\ell}}|^{q},
    \end{split}
\end{equation*}   
where $C= C(d,s,p,q,k, \alpha, \beta)$ is a positive constant. Summing the above inequality from $i=1$ to $\sigma_{\ell}$, and then from  $\ell= m \in \mathbb{Z}^{-}$ to $n_{0}$, and applying \eqref{sumineq} with $\gamma= \frac{q}{p}$, we obtain
\begin{equation}\label{eqnn1}
\begin{split}
\sum_{\ell=m}^{n_{0}} & \int_{\mathcal{B}_{\ell}} \frac{|u(x)|^{q}}{|x_{k}|^{\theta + \left( k-sp \right) \frac{q}{p} }} dx \\ & \leq C \sum_{\ell=m}^{n_{0}}  \left( \int_{\mathcal{B}_{\ell}} \int_{\mathcal{B}_{\ell}} \frac{|u(x)-u(y)|^{p}}{|x-y|^{d+sp}} |x_{k}|^{-(k-\alpha+\beta-sp)/2} |y_{k}|^{-(k+\alpha-\beta-sp)/2} \, dx \, dy \right)^{\frac{q}{p}} \\ & \hspace{5mm} + C \sum_{\ell=m}^{n_{0}} 2^{\ell \left( d-k \right) \frac{q}{p}}  \sum_{i=1}^{\sigma_{\ell}} |(u)_{\mathcal{B}^{i}_{\ell}}|^{q}.
\end{split}
\end{equation} 
Consider the disjoint sets $\mathcal{B}^{j}_{\ell+1}$ and $\mathcal{B}^{i}_{\ell}$, where $\mathcal{B}^{j}_{\ell+1}$ is chosen so that $C^{i}_{\ell}$ is contained in $C^{j}_{\ell+1}$. In other words,   
\begin{equation}\label{codn on Al and Al+1}
    \mathcal{B}^{i}_{\ell} \subset \{ (x_{k}, x_{d-k}) \in \mathcal{B}_{\ell} : x_{d-k} \in C^{j}_{\ell +1}   \}.
\end{equation}
Moreover, there exist $2^{d-k}$ such sets $\mathcal{B}^{i}_{\ell}$ satisfying this condition. Applying \cite[Lemma 3.2]{Sahu2025}, we obtain  
\begin{equation}\label{estimate on B_l and B_l+1}
 |(u)_{\mathcal{B}^{i}_{\ell}} - (u)_{\mathcal{B}^{j}_{\ell+1}} |^{q} \leq C 2^{\ell \left(sp-d \right) \frac{q}{p}} [u]^{q}_{W^{s, p}(\mathcal{B}^{i}_{\ell} \cup \mathcal{B}^{j}_{\ell+1})},
    \end{equation}
where $C$ is a positive constant independent of $\ell$. Using triangle inequality, we have 
\begin{equation*}
    |(u)_{\mathcal{B}^{i}_{\ell}}|^q \leq  \left( |(u)_{\mathcal{B}^{j}_{\ell+1}}| + |(u)_{\mathcal{B}^{i}_{\ell}} - (u)_{\mathcal{B}^{j}_{\ell+1}}| \right)^q.
\end{equation*}
We know that $k<d$ and $q>p$. Hence, applying Lemma \ref{Lemma: estimate} with $c :=c_{1}2^{(d-k)\frac{q}{p} -(d-k) }>1$ where $c_{1} = \frac{2}{1+ 2^{(d-k)\frac{q}{p} -(d-k)}} <1$ together with \eqref{estimate on B_l and B_l+1}, we obtain
\begin{equation*}
    |(u)_{\mathcal{B}^{i}_{\ell}}|^{q} \leq  c_{1}2^{(d-k)\frac{q}{p} -(d-k) } |(u)_{\mathcal{B}^{j}_{\ell+1}}|^{q} + C 2^{\ell(sp-d)\frac{q}{p}}  [u]^{q}_{W^{s, p}(\mathcal{B}^{i}_{\ell} \cup \mathcal{B}^{j}_{\ell+1})}  .
\end{equation*}
Multiplying the above inequality by $2^{\ell \left( d-k \right) \frac{q}{p}}$, we get
\begin{equation*}
   2^{\ell \left( d-k \right) \frac{q}{p}} |(u)_{\mathcal{B}^{i}_{\ell}}|^{q} \leq c_{1} 2^{(\ell+1) \left( \left( d-k \right) \frac{q}{p} \right) - (d-k)}|(u)_{\mathcal{B}^{j}_{\ell+1}}|^{q} +  C 2^{\ell \left( -k+sp \right) \frac{q}{p}} [u]^{p}_{W^{s, p}(\mathcal{B}^{i}_{\ell} \cup \mathcal{B}^{j}_{\ell+1})}  .
\end{equation*}
Since there are $2^{d-k}$ such sets $\mathcal{B}^{i}_{\ell}$ satisfying \eqref{codn on Al and Al+1}, we sum the above inequality over $i$ from $2^{d-k}(j-1)+1$ to $2^{d-k}j$. Applying the relation
\begin{equation*}
2^{\ell \left( -k+sp \right)}= 2^{\ell (-k+\alpha-\beta + sp)/2 } 2^{\ell (-k-\alpha+\beta + sp)/2 } \leq C |x_{k}|^{-(k-\alpha+\beta-sp)/2} |y_{k}|^{-(k+\alpha-\beta-sp)/2}
\end{equation*} 
for some positive constant $C=C(s,p,k,\alpha,\beta)$, valid for all $(x_{k}, x_{d-k}), (y_{k}, y_{d-k}) \in \mathcal{B}^{i}_{\ell} \cup \mathcal{B}^{j}_{\ell+1}$, we then sum from $j=1$ to $\sigma_{\ell+1}$ and apply \eqref{sumineq}  with $\gamma= \frac{q}{p}$ to obtain
\begin{equation*}\label{ineqn100}
\begin{split}
   & 2^{\ell \left( d-k \right) \frac{q}{p}} \sum_{i=1}^{\sigma_{\ell}} |(u)_{\mathcal{B}^{i}_{\ell}}|^{q} \\ & \leq c_{1} 2^{(\ell+1) \left( \left( d-k \right) \frac{q}{p} \right)} \sum_{j=1}^{\sigma_{\ell+1}} |(u)_{\mathcal{B}^{j}_{\ell+1}}|^{q} \\ & \hspace{4mm}   + \ C \left( \int_{\mathcal{B}_{\ell} \cup \mathcal{B}_{\ell+1}} \int_{\mathcal{B}_{\ell} \cup \mathcal{B}_{\ell+1}} \frac{|u(x)-u(y)|^{p}}{|x-y|^{d+sp}} |x_{k}|^{-(k-\alpha+\beta-sp)/2} |y_{k}|^{-(k+\alpha-\beta-sp)/2} \, dx \, dy \right)^{\frac{q}{p}}.
\end{split}
\end{equation*}
Next, summing the resulting inequality from $\ell = m \in \mathbb{Z}^{-}$ to $n_{0}$, and then rearranging terms with appropriate re-indexing, we obtain
\begin{equation*}
\begin{split}
 &  2^{m \left( d-k \right) \frac{q}{p}} \sum_{i=1}^{\sigma_{m}} |(u)_{\mathcal{B}^{i}_{m}}|^{q} + (1-c_{1}) \sum_{\ell=m+1}^{n_{0}} 2^{\ell \left( d-k \right) \frac{q}{p}} \sum_{i=1}^{\sigma_{\ell}} |(u)_{\mathcal{B}^{i}_{\ell}}|^{q}  \\ & \leq  C \sum_{j=1}^{\sigma_{n_{0}+1}} |(u)_{\mathcal{B}^{j}_{n_{0}+1}}|^{q} 
    \\ & \hspace{3mm} +   C \sum_{\ell=m}^{n_{0}} \left( \int_{\mathcal{B}_{\ell} \cup \mathcal{B}_{\ell+1}} \int_{\mathcal{B}_{\ell} \cup \mathcal{B}_{\ell+1}} \frac{|u(x)-u(y)|^{p}}{|x-y|^{d+sp}} |x_{k}|^{-(k-\alpha+\beta-sp)/2} |y_{k}|^{-(k+\alpha-\beta-sp)/2} \, dx \, dy \right)^{\frac{q}{p}} .
    \end{split}
\end{equation*}
Since $\operatorname{supp} \, u \subset \mathcal{D}$, this implies that $|(u)_{\mathcal{B}^{j}_{n_{0}+1}}|=0$ for all $j \in \{ 1, \dots, \sigma_{n_{0}+1} \}$. Hence, we have for some constant $C= C(d,s,p,q,k, \alpha, \beta)>0$,
\begin{equation}\label{eqnn99}
\begin{split}
   & (1-c_{1})  \sum_{\ell=m}^{n_{0}} 2^{\ell \left( d-k \right) \frac{q}{p}} \sum_{i=1}^{\sigma_{\ell}} |(u)_{\mathcal{B}^{i}_{\ell}}|^{q}  \\ & \leq    C \sum_{\ell=m}^{n_{0}} \left( \int_{\mathcal{B}_{\ell} \cup \mathcal{B}_{\ell+1}} \int_{\mathcal{B}_{\ell} \cup \mathcal{B}_{\ell+1}} \frac{|u(x)-u(y)|^{p}}{|x-y|^{d+sp}} |x_{k}|^{-(k-\alpha+\beta-sp)/2} |y_{k}|^{-(k+\alpha-\beta-sp)/2} \, dx \, dy \right)^{\frac{q}{p}}.
    \end{split}
\end{equation}
Combining \eqref{eqnn1} and \eqref{eqnn99} yields
\begin{equation*}
\begin{split}
   & \sum_{\ell=m}^{n_{0}}  \int_{\mathcal{B}_{\ell}} \frac{|u(x)|^{q}}{|x_{k}|^{\theta + \left( k-sp \right) \frac{q}{p} }} \, dx \\ & \leq  C \sum_{\ell=m}^{n_{0}} \left( \int_{\mathcal{B}_{\ell} \cup \mathcal{B}_{\ell+1}} \int_{\mathcal{B}_{\ell} \cup \mathcal{B}_{\ell+1}} \frac{|u(x)-u(y)|^{p}}{|x-y|^{d+sp}} |x_{k}|^{-(k-\alpha+\beta-sp)/2} |y_{k}|^{-(k+\alpha-\beta-sp)/2} \, dx \, dy \right)^{\frac{q}{p}}.
    \end{split}
\end{equation*}
 By applying \eqref{sumineq} with $\gamma= \frac{q}{p}$ to the right-hand side of the above inequality and utilizing \cite[Lemma 3.3]{Sahu2025}, we conclude the first part of the proof of Theorem \ref{Theorem : Weighted fractional Hardy  sp leq d}.

 \smallskip

For the second part of the proof of Theorem \ref{Theorem : Weighted fractional Hardy  sp leq d}, we proceed similarly to the first part with $p=2$ and $r>1$. We utilize the fact that for any $\alpha, \beta \in \mathbb{R}$, we have
\begin{equation*}
\begin{split}
2^{\ell \left( -k+2s \right)} & = 2^{\ell \left(- \frac{k+\alpha+\beta-2s}{r} \right) } 2^{\ell \left(- \frac{k+\alpha+\beta-2s}{r} \right)(r-1) } 2^{\ell \alpha} 2^{\ell \beta} \\ & \leq C \min \left\{ |x_{k}|^{- \frac{k+\alpha+\beta-2s}{r}}, |y_{k}|^{- \frac{k+\alpha+\beta-2s}{r}} \right\} \max \left\{ |x_{k}|^{- \frac{k+\alpha+\beta-2s}{r}}, |y_{k}|^{- \frac{k+\alpha+\beta-2s}{r}} \right\}^{r-1} \\ & \hspace{5mm} \times |x_{k}|^{\alpha} |y_{k}|^{\beta} \\ & = CW_{r}(x,y) |x_{k}|^{\alpha} |y_{k}|^{\beta} ,
\end{split}
\end{equation*}  
for some positive constant $C=C(s,r,k,\alpha,\beta)$, and for all $(x_k, x_{d-k}), (y_k, y_{d-k}) \in \mathcal{B}^{i}_{\ell}$. Here, $W_{r}(x,y)$ is as defined in \eqref{Defn: W_r}. Therefore, for any $q' \in \left(2, \frac{2d}{d-2s} \right]$, we obtain
\begin{equation*}
  \left(  \int_{\mathbb{R}^{d}} \int_{\mathbb{R}^{d}} \frac{|u(x)-u(y)|^{2}}{|x-y|^{d+2s}} W_{r}(x,y) |x_{k}|^{\alpha} |y_{k}|^{\beta} \, dx \, dy \right)^{\frac{1}{2}} \geq C \left( \int_{\mathbb{R}^{d}} \frac{|u(x)|^{q'}}{|x_{k}|^{\theta' + \left( k-2s \right) \frac{q'}{2} }} \, dx \right)^{\frac{1}{q'}}. 
\end{equation*}
This completes the second part of the proof.
\end{proof}

\begin{remark}
 For $sp<d$, Theorem \ref{Theorem : Weighted fractional Hardy  sp leq d} could also be proved using the inequality from \cite[Theorem 3.1]{MR3803664} and the fact that, for $2\leq k<d$, the set $\R^d\setminus K$ is a John domain. However, the mentioned inequality does not capture the endpoint case $sp=d$ and therefore we decided to tackle a different strategy of the proof.
\end{remark}

\section{The case \texorpdfstring{$k=d$}{k=d}}\label{The case k=d} 

In this section, we examine the case $k = d$, which corresponds to a  singularity at the origin, and prove Theorem \ref{Theorem 6}, which establishes a weighted fractional Hardy--Sobolev--Maz'ya inequality for the case $k=d$ with an appropriate logarithmic weight function. Our primary objective is to first derive a weighted fractional Hardy--Sobolev inequality with a logarithmic weight function. The inclusion of the logarithmic function arises naturally due to the integrability condition near the origin. We also show that, when $k=d$, there is no chance to have the Hardy--Sobolev--Maz'ya inequality like in \eqref{HSMk<d}, i.e. we show that such inequality fails in general.

\smallskip  

The following theorem establishes a weighted fractional Hardy--Sobolev inequality involving a logarithmic weight function, which plays a fundamental role in deriving the weighted fractional Hardy--Sobolev--Maz'ya inequality for the case $k=d$. The proof is inspired by the ideas presented in \cite[Theorem 3.1]{Nguyen2018}, where a similar result is obtained in the setting of fractional Caffarelli--Kohn--Nirenberg inequalities for the critical case. Specifically, we formulate and prove the following theorem.

\begin{theorem}\label{Th : Weighted fractional Hardy k=d}
Let $d \geq 1$, $p > 1$, $q > 1$, $\alpha, \beta \in \mathbb{R}$, and $s \in (0,1)$. If $sp < d$, assume that $p \leq q \leq \frac{dp}{d - sp}$, and if $sp = d$, assume that $q \geq p$. Then, there exists a constant $C = C(d, s, p, q, \alpha, \beta) > 0$ such that for all $u \in C^{1}_{c}(\mathbb{R}^{d})$ with $\operatorname{supp} u \subset B(0,R)$ for some $R>0$, the following inequality holds:
\begin{equation}\label{Th 4.1 : ineq 1}
\begin{split}
     \Bigg( \int_{\mathbb{R}^{d}} \int_{\mathbb{R}^{d}} \frac{|u(x)-u(y)|^{p}}{|x-y|^{d+sp}}  |x|^{-(d-\alpha+\beta-sp)/2} &  |y|^{-(d+\alpha-\beta-sp)/2}  \, dx \,dy  \Bigg)^{\frac{1}{p}} \\ & \hspace{5mm} \geq C \left( \int_{\mathbb{R}^{d}} \frac{|u(x)|^{q}}{|x|^{d} \ln^{q} \left( \frac{4R}{|x|} \right)} \, dx \right)^{\frac{1}{q}}.
\end{split}  
\end{equation}
Furthermore, for any $q'$ satisfying $2 \leq q' \leq \frac{2d}{d-2s}$ when $2s<d$, and $q' \geq 2$ when $2s=d$, and for any $r>1$, we have for any $u \in C^{1}_{c}(\mathbb{R}^{d})$ with $\operatorname{supp} u \subset B(0,R)$ for some $R>0$,
\begin{equation}\label{Th 4.1 : ineq 2}
  \left(  \int_{\mathbb{R}^{d}} \int_{\mathbb{R}^{d}} \frac{|u(x)-u(y)|^{2}}{|x-y|^{d+2s}} W^{d}_{r}(x,y) |x|^{\alpha} |y|^{\beta} \, dx \, dy \right)^{\frac{1}{2}} \geq C \left( \int_{\mathbb{R}^{d}} \frac{|u(x)|^{q'}}{|x|^{d} \ln^{q'} \left( \frac{4R}{|x|} \right)} \, dx \right)^{\frac{1}{q'}}. 
\end{equation}
where
\begin{equation}\label{Defn W rd}
    W^{d}_{r}(x,y)= \min \left\{ |x|^{- \frac{d+\alpha+\beta-2s}{r}}, |y|^{- \frac{d+\alpha+\beta-2s}{r}} \right\} \max \left\{ |x|^{- \frac{d+\alpha+\beta-2s}{r}}, |y|^{- \frac{d+\alpha+\beta-2s}{r}} \right\}^{r-1},
\end{equation}
and $C=C(d,s,r,q',\alpha, \beta)$ is a positive constant.
\end{theorem}
\begin{proof}
Let $u \in C^{1}_{c}(\mathbb{R}^{d})$ such that $\operatorname{supp} u \subset B(0,R)$, where $2^{n} \leq R <2^{n+1}$ for some $n \in \mathbb{Z}$. For each $\ell \in \mathbb{Z}$, define
\begin{equation*}
    \mathcal{A}_{\ell} := \left\{ x \in \mathbb{R}^{d} : 2^{\ell} \leq |x| < 2^{\ell+1} \right\}.
\end{equation*}
 By applying Lemma \ref{sobolev} with the domain $\Omega = \{  x \in \mathbb{R}^{d} : 1 < |x| < 2 \}$, $ \lambda = 2^{\ell}$, and using translation invariance, we have
\begin{equation*}
 \fint_{\mathcal{A}_{\ell}} |u(x)-(u)_{\mathcal{A}_{\ell}}|^{q}  dx \leq C 2^{\ell \left(sp-d \right) \frac{q}{p} }[u]^{q}_{W^{s, p}(\mathcal{A}_{\ell})}  ,
\end{equation*}
where $C= C(d,p,s,q)$ is a positive constant. Furthermore, we have
\begin{equation}\label{ineq 1 k=d}
2^{\ell \left( sp-d \right)}= 2^{\ell (-d+\alpha-\beta + sp)/2 } 2^{\ell (-d-\alpha+\beta + sp)/2 } \leq C |x|^{-(d-\alpha+\beta-sp)/2} |y|^{-(d+\alpha-\beta-sp)/2}
\end{equation}  
for some positive constant $C=C(s,p,d,\alpha,\beta)$, for all $x,y \in \mathcal{A}_{\ell}$. Substituting this into the previous inequality, we obtain
\begin{equation*}
\begin{split}
    \fint_{\mathcal{A}_{\ell}} |u(x)- &(u)_{\mathcal{A}_{\ell}}|^{q}  dx  \\ & \leq C \left( \int_{\mathcal{A}_{\ell}} \int_{\mathcal{A}_{\ell}} \frac{|u(x)-u(y)|^{p}}{|x-y|^{d+sp}} |x|^{-(d-\alpha+\beta-sp)/2} |y|^{-(d+\alpha-\beta-sp)/2} \, dx \, dy \right)^{\frac{q}{p}}.
\end{split} 
\end{equation*}
For each $x \in \mathcal{A}_{\ell}$, we have $|x| < 2^{\ell+1}$, which leads to the inequality $\ln \left(  \frac{4R}{|x|} \right) \geq n-\ell+1$. By making use of this inequality, along with $\frac{1}{(n-\ell+1)^{q}} \leq 1$ for all $\ell \leq n$, and using the above inequality together with triangle inequality, we obtain
\begin{equation*}
    \begin{split}
        \int_{\mathcal{A}_{\ell}} \frac{|u(x)|^{q}}{|x|^{d} \ln^{q} \left( \frac{4R}{|x|} \right)} \, dx & \leq C \left( \int_{\mathcal{A}_{\ell}} \int_{\mathcal{A}_{\ell}} \frac{|u(x)-u(y)|^{p}}{|x-y|^{d+sp}} |x|^{-(d-\alpha+\beta-sp)/2} |y|^{-(d+\alpha-\beta-sp)/2} \, dx \, dy \right)^{\frac{q}{p}} \\ & \hspace{5mm} + C \frac{|(u)_{\mathcal{A}_{\ell}}|^{q}}{(n-\ell+1)^{q}},
    \end{split}
\end{equation*}
where $C=C(s,p,d, \alpha, \beta, q)$ is a positive constant. By summing the above inequality from $\ell=m$ to $n$, we obtain
\begin{equation}\label{ineq 2 k=d}
    \begin{split}
      \sum_{\ell=m}^{n}  \int_{\mathcal{A}_{\ell}}  & \frac{|u(x)|^{q}}{|x|^{d} \ln^{q} \left( \frac{4R}{|x|} \right)} \, dx  \\ &  \leq C \sum_{\ell=m}^{n} \left( \int_{\mathcal{A}_{\ell}} \int_{\mathcal{A}_{\ell}} \frac{|u(x)-u(y)|^{p}}{|x-y|^{d+sp}} |x|^{-(d-\alpha+\beta-sp)/2} |y|^{-(d+\alpha-\beta-sp)/2} \, dx \, dy \right)^{\frac{q}{p}} \\ & \hspace{5mm} + C \sum_{\ell=m}^{n}  \frac{|(u)_{\mathcal{A}_{\ell}}|^{q}}{(n-\ell+1)^{q}}.
    \end{split}
\end{equation}
Applying \cite[Lemma 3.2]{Sahu2025} and using the inequality \eqref{ineq 1 k=d} for all $x,y \in  \mathcal{A}_{\ell} \cup \mathcal{A}_{\ell+1}$, we obtain  
\begin{equation*}
\begin{split}
   |(u)_{\mathcal{A}_{\ell}} & - (u)_{\mathcal{A}_{\ell+1}} |^{q}  \leq C 2^{\ell \left(sp-d \right) \frac{q}{p}} [u]^{q}_{W^{s, p}(\mathcal{A}_{\ell} \cup \mathcal{A}_{\ell+1})} \\ & \leq C \left( \int_{\mathcal{A}_{\ell} \cup \mathcal{A}_{\ell+1}} \int_{\mathcal{A}_{\ell} \cup \mathcal{A}_{\ell+1}} \frac{|u(x)-u(y)|^{p}}{|x-y|^{d+sp}} |x|^{-(d-\alpha+\beta-sp)/2} |y|^{-(d+\alpha-\beta-sp)/2} \, dx \, dy \right)^{\frac{q}{p}},
\end{split}
    \end{equation*}
where $C$ is a positive constant independent of $\ell$. By applying Lemma \ref{Lemma: estimate} with $c :=\left( \frac{n-\ell+1}{n-\ell+1/2} \right)^{q-1}$ together with the above inequality, we obtain
\begin{equation*}
\begin{split}
   & \frac{|(u)_{\mathcal{A}_{\ell}} |^{q}}{(n-\ell+1)^{q-1}}  \leq \frac{|(u)_{\mathcal{A}_{\ell+1}}|^{q}}{(n-\ell+ 1/2)^{q-1}} \\ & \hspace{8mm} + C \left( \int_{\mathcal{A}_{\ell} \cup \mathcal{A}_{\ell+1}} \int_{\mathcal{A}_{\ell} \cup \mathcal{A}_{\ell+1}} \frac{|u(x)-u(y)|^{p}}{|x-y|^{d+sp}} |x|^{-(d-\alpha+\beta-sp)/2} |y|^{-(d+\alpha-\beta-sp)/2} \, dx \, dy \right)^{\frac{q}{p}}.
\end{split}  
\end{equation*}
Summing the above inequality from $\ell=m$ to $n$, and using the asymptotics
\begin{equation*}
    \frac{1}{(n-\ell+1)^{q-1}} - \frac{1}{(n-\ell + 3/2)^{q-1}} \sim \frac{1}{(n-\ell+1)^{q}},
\end{equation*}
and combining with the inequality \eqref{ineq 2 k=d}, we obtain
\begin{equation*}
\begin{split}
    \sum_{\ell=m}^{n} &  \int_{\mathcal{A}_{\ell}}   \frac{|u(x)|^{q}}{|x|^{d} \ln^{q} \left( \frac{4R}{|x|} \right)} \, dx  \\ &  \leq C \sum_{\ell=m}^{n} \left( \int_{\mathcal{A}_{\ell} \cup \mathcal{A}_{\ell+1}} \int_{\mathcal{A}_{\ell} \cup \mathcal{A}_{\ell+1}} \frac{|u(x)-u(y)|^{p}}{|x-y|^{d+sp}} |x|^{-(d-\alpha+\beta-sp)/2} |y|^{-(d+\alpha-\beta-sp)/2} \, dx \, dy \right)^{\frac{q}{p}}.
\end{split}
\end{equation*}
By applying \eqref{sumineq} with $\gamma= \frac{q}{p}$ to the right-hand side of the above inequality and utilizing \cite[Lemma 3.3]{Sahu2025}, we prove the first part of the proof of theorem.

For the second part of the proof of theorem, we proceed similarly to the first part with $p=2$ and $r>1$. We utilize the fact that for any $\alpha, \beta \in \mathbb{R}$, we have
\begin{equation*}
\begin{split}
2^{\ell \left( 2s-d \right)} & = 2^{\ell \left(- \frac{d+\alpha+\beta-2s}{r} \right) } 2^{\ell \left(- \frac{d+\alpha+\beta-2s}{r} \right)(r-1) } 2^{\ell \alpha} 2^{\ell \beta} \\ & \leq C \min \left\{ |x|^{- \frac{d+\alpha+\beta-2s}{r}}, |y|^{- \frac{d+\alpha+\beta-2s}{r}} \right\} \max \left\{ |x|^{- \frac{d+\alpha+\beta-2s}{r}}, |y|^{- \frac{d+\alpha+\beta-2s}{r}} \right\}^{r-1}  |x|^{\alpha} |y|^{\beta} \\ & = CW^{d}_{r}(x,y) |x|^{\alpha} |y|^{\beta} ,
\end{split}
\end{equation*}  
for some positive constant $C=C(s,r,d,\alpha,\beta)$, and for all $x,y \in \mathcal{A}_{\ell}$. Here, $W^{d}_{r}(x,y)$ is as defined in \eqref{Defn W rd}. Therefore, for any $q' \in \left[ 2, \frac{2d}{d-2s} \right]$ when $2s<d$, and $q' \geq 2$ when $2s=d$, we obtain
\begin{equation*}
  \left(  \int_{\mathbb{R}^{d}} \int_{\mathbb{R}^{d}} \frac{|u(x)-u(y)|^{2}}{|x-y|^{d+2s}} W^{d}_{r}(x,y) |x|^{\alpha} |y|^{\beta} \, dx \, dy \right)^{\frac{1}{2}} \geq C \left( \int_{\mathbb{R}^{d}} \frac{|u(x)|^{q'}}{|x|^{d} \ln^{q'} \left( \frac{4R}{|x|} \right) } \, dx \right)^{\frac{1}{q'}}. 
\end{equation*}
This finishes the second part of the proof.
\end{proof}

\begin{proof}[\textbf{Proof of Theorem \ref{Theorem 6}}] 
We begin by defining the function
\begin{equation*}
  v(x) = |x|^{(d + \alpha + \beta - sp)/p} u(x).  
\end{equation*}
Using this definition, we obtain the following relation:
\begin{equation}
\left( \int_{\mathbb{R}^{d}} \frac{|v(x)|^{q}}{|x|^{d}\ln^{q} \left( \frac{4R}{|x|} \right)} \, dx \right)^{\frac{1}{q}} = \left( \int_{\mathbb{R}^{d}} \frac{|u(x)|^{q}}{\ln^{q} \left( \frac{4R}{|x|} \right)}  |x|^{\frac{q}{p}(\alpha+\beta) -\theta} \, dx \right)^{\frac{1}{q}}.
\end{equation}
By substituting this into the inequality \eqref{Th 4.1 : ineq 1} given in Theorem \ref{Th : Weighted fractional Hardy k=d},  and subsequently applying \cite[Theorem 1.6]{dyda2022sharp}, we establish the weighted fractional Hardy–Sobolev–Maz'ya inequality for the case $p \geq 2$.

For the case $1<p<2$, we proceed by using inequality \eqref{Th 4.1 : ineq 2} from Theorem \ref{Th : Weighted fractional Hardy k=d}. By making the appropriate substitutions, replacing $s$ with $\frac{sp}{2}$ (noting that for $1<p<2$, $\frac{sp}{2} \in (0,1)$), and setting $r=p$, we obtain that
\begin{equation*}
\begin{split}
    \int_{\mathbb{R}^{d}} \int_{\mathbb{R}^{d}} \frac{\left( v(x)^{\langle p/2 \rangle}-v(y) )^{\langle p/2 \rangle} \right)^{2}}{|x-y|^{d+sp}} W^{d}(x,y) &|x|^{\alpha} |y|^{\beta}  \, dx \, dy \\ & \geq C   \left( \int_{\mathbb{R}^{d}} \frac{|u(x)|^{q}}{\ln^{q} \left( \frac{4R}{|x|} \right)}  |x|^{\frac{q}{p}(\alpha+\beta) -\theta} \, dx \right)^{\frac{p}{q}},
\end{split}
\end{equation*}
where  $p \leq q \leq \frac{dp}{d - sp}$ for $sp<d$, and $p \geq q$ for $sp=d$, and 
\begin{equation*}
  W^{d}(x,y)= \min \left\{ |x|^{- \frac{d+\alpha+\beta-sp}{p}}, |y|^{- \frac{d+\alpha+\beta-sp}{p}} \right\} \max \left\{ |x|^{- \frac{d+\alpha+\beta-sp}{p}}, |y|^{- \frac{d+\alpha+\beta-sp}{p}} \right\}^{p-1}.  
\end{equation*}
Finally, by employing \cite[Theorem 2]{dyda2024} along with the above inequality, we establish the weighted fractional Hardy–Sobolev–Maz'ya inequality for the case $1<p<2$.  This completes the proof.
\end{proof}

\smallskip

The following Theorem provides a counterexample showing that the Hardy--Sobolev--Maz'ya inequality \eqref{HSMk<d} does not hold, when $k=d$.
\begin{theorem}\label{Theorem : HSM fails k=d}
Let $\al,\be,\al+\be\in(-d,sp)$ and $0<s<1$, $d,p\geq1$, $\al+\be+sp\neq d$, $sp<d$, $q=dp/(d-sp)$. Then, the weighted Hardy--Sobolev--Maz'ya inequality fails, i.e. it \textbf{does not} exist a positive constant $C_0=C_0(d,s,p,\al,\be)$ such that
\begin{equation}\label{HSM}
   \begin{split}
     \int_{\mathbb{R}^{d}} \int_{\mathbb{R}^{d}}  \frac{|u(x)-u(y)|^{p}}{|x-y|^{d+sp}}  |x|^{\alpha} |y|^{\beta} \, dy \, dx - \mathcal{C} &  \int_{\mathbb{R}^{d}} \frac{|u(x)|^{p}}{|x|^{sp-\alpha-\beta}} \, dx  \\ & \geq C_0 \left( \int_{\mathbb{R}^{d}} |u(x)|^{q} |x|^{ \frac{q}{p}(\alpha+\beta)} \, dx \right)^{\frac{p}{q}}  
     \end{split}
    \end{equation} 
    for all $u\in W^{s,p}_{\al,\be}(\R^d)$, where $\mathcal{C}=\mathcal{C}_1$ is the optimal constant in the weighted fractional Hardy inequality given by \eqref{The value C 1}.
\end{theorem}
\begin{proof}
    We provide the proof only for $d>1$; the proof for $d=1$ is similar. For $u\in W^{s,p}_{\al,\be}(\R^d)$, we consider the ratio
    \begin{equation}\label{psi}
    \Psi(u)=\frac{\displaystyle\int_{\R^d}\int_{\R^d}\frac{|u(x)-u(y)|^p}{|x-y|^{d+sp}}|x|^{\al}|y|^{\be}\,dy\,dx-\mathcal{C}\int_{\R^d}\frac{|u(x)|^p}{|x|^{sp-\al-\be}}\,dx}{\left(\displaystyle\int_{\R^d}|u(x)|^q|x|^{q(\al+\be)/p}\,dx\right)^{p/q}}\geq 0.
       \end{equation}
    Suppose that there exists a positive constant $C_0=C_0(d,s,p,\al,\be)$ such that $\Psi(u)\geq C_0$ for all $u\in W^{s,p}_{\al,\be}(\R^d)$. We will show that this is impossible by constructing a set of functions $\{u_{\varepsilon}\}_{\varepsilon>0}\subset W^{s,p}_{\al,\be}(\R^d)$ such that $\Psi(u_{\varepsilon})\rightarrow 0$, when $\varepsilon\rightarrow 0^+$. To this end, fix $\varepsilon>0$ and let $\ga=(d+\al+\be-sp)/p$. Suppose first that $\ga>0$. We define the functions
    $$
    u_{\varepsilon}(x)=\begin{cases}
        |x|^{\ga} & \text{ for }x\in B, \\ 
        |x|^{-\ga-\varepsilon} & \text{ for }x\in B^c,
    \end{cases}
    $$
    where $B$ is the unit ball in $\R^d$, centered at the origin. A direct calculation yields
    \begin{align*}
    \int_{\R^d}\frac{|u(x)|^p}{|x|^{sp-\al-\be}}\,dx&=\int_B|x|^{d+2(\al+\be-sp)}\,dx+\int_{B^c}|x|^{-d-p\varepsilon}\,dx\\
    &=\meas\left(\frac{1}{2(d+\al+\be-sp)}+\frac{1}{p\varepsilon}\right)
    \end{align*}
    and, recalling that $q=\frac{dp}{d-sp}$,
    \begin{align}\label{p/q}
\nonumber\left(\displaystyle\int_{\R^d}|u(x)|^q|x|^{\frac{q(\al+\be)}{p}}\,dx\right)^{\frac{p}{q}}&\geq\left(\int_{B^c}|x|^{-\ga q-q\varepsilon+\frac{q(\al+\be)}{p}}\,dx\right)^{\frac{p}{q}}\\
\nonumber&=\left(\int_{B^c}|x|^{-d-q\varepsilon}\,dx\right)^{\frac{p}{q}}\\
&=\left(\frac{\meas}{q}\right)^{\frac{p}{q}}\varepsilon^{-\frac{p}{q}}.
    \end{align}
    Now, we will estimate the weighted Gagliardo seminorm of $u_{\varepsilon}$. We have
    \begin{align*}
    [u]^p_{W^{s,p}_{\al,\be}(\R^d)}&=\int_B\int_B\frac{||x|^{\ga}-|y|^{\ga}|^p}{|x-y|^{d+sp}}|x|^{\al}|y|^{\be}\,dy\,dx\\
    &\quad+\int_{B^c}\int_{B^c}\frac{||x|^{-\ga-\varepsilon}-|y|^{-\ga-\varepsilon}|^p}{|x-y|^{d+sp}}|x|^{\al}|y|^{\be}\,dy\,dx\\
    &\quad+2\int_B\int_{B^c}\frac{||x|^{\ga}-|y|^{-\ga-\varepsilon}|^p}{|x-y|^{d+sp}}|x|^{\al}|y|^{\be}\,dy\,dx.
    \end{align*}
We will begin with the integral over $B^c\times B^c$. Using the substitution (inversion) $x\mapsto T(x)=x/|x|^2$, $y\mapsto T(y)=y/|y|^2$ with Jacobian $(|x||y|)^{-2d}$, the fact that $|T(x)-T(y)|=|x-y|/(|x||y|)$, symmetry and rotation invariance of the inner integral, we get
\begin{align*}
 \int_{B^c}\int_{B^c}&\frac{||x|^{-\ga-\varepsilon}-|y|^{-\ga-\varepsilon}|^p}{|x-y|^{d+sp}}|x|^{\al}|y|^{\be}\,dy\,dx\\
 &=\frac{1}{2}\int_{B^c}\int_{B^c}\frac{||x|^{-\ga-\varepsilon}-|y|^{-\ga-\varepsilon}|^p}{|x-y|^{d+sp}}\left(|x|^{\al}|y|^{\be}+|x|^{\be}|y|^{\al}\right)\,dy\,dx\\
 &=\frac{1}{2}\int_B\int_B\frac{||x|^{\ga+\varepsilon}-|y|^{\ga+\varepsilon}|^p}{|x-y|^{d+sp}}\left(|x|^{sp-\al-d}|y|^{sp-\be-d}+|x|^{sp-\be-d}|y|^{sp-\al-d}\right)\,dy\,dx\\
 &=\int_B\int_{|y|\leq|x|}\frac{||x|^{\ga+\varepsilon}-|y|^{\ga+\varepsilon}|^p}{||x|e_d-y|^{d+sp}}\left(|x|^{sp-\al-d}|y|^{sp-\be-d}+|x|^{sp-\be-d}|y|^{sp-\al-d}\right)\,dy\,dx\\
 &=\int_B|x|^{-d+p\varepsilon}\,dx\int_B
\frac{|1-|z|^{\ga+\varepsilon}|^p}{|e_d-z|^{d+sp}}|z|^{sp-d}\left(|z|^{-\al}+|z|^{-\be}\right)\,dz\\
&=\frac{\meas}{p\varepsilon}\times\left|\mathbb{S}^{d-2}\right|\int_{0}^{1}\int_{-1}^{1}\frac{|1-r^{\ga+\varepsilon}|^p r^{sp-1}(r^{-\al}+r^{-\be})(1-t^2)^{\frac{d-3}{2}}}{(1-2rt+r^2)^{(d+sp)/2}}\,dt\,dr\\
&=:\frac{\meas}{p\varepsilon}\mathcal{C}_{\varepsilon},
\end{align*}
where $e_d=(0,0,\dots,0,1)$. Observe that $0<\gamma<\gamma+\varepsilon$ implies $|1-t^{\gamma+\varepsilon}|^p>|1-t^{\gamma}|^p$ for all $t\in(0,1)$ and therefore $\mathcal{C}_{\varepsilon}>\mathcal{C}$.

The integral over $B\times B$ is finite, since the function $|x|^{\ga}$, $x\in B$, can be extended to a $C_c^1(\R^d)$ function and such functions belong to $W^{s,p}_{\al,\be}(\R^d)$ by \cite[Lemma 2.1]{valdinoci2015}, and independent on $\varepsilon$. Finally, we need to deal with the integral over $B\times B^c$. We want to show that this integral is bounded for $\varepsilon\rightarrow 0^+$, hence, it suffices to show the finiteness of the integral
$$
I:=\int_B\int_{B^c}\frac{||x|^{\ga}-|y|^{-\ga}|^p}{|x-y|^{d+sp}}|x|^{\al}|y|^{\be}\,dy\,dx.
$$
Again, using the rotation invariance of the inner integral, we have
\begin{align*}
I&=\meas\int_{0}^1r^{d+\al-1}\,dr\int_{B^c}\frac{|r^{\ga}-|y|^{-\ga}|^p}{|re_d-y|^{d+sp}}|y|^{\be}\,dy. 
\end{align*}
Substituting $y=\frac{z}{r}$ with $dy=r^{-d}dz$ in the inner integral, after some manipulations we get
\begin{equation}\label{I}
 I=\meas\int_{0}^{1}r^{2(d+\al)-1}\,dr\int_{|z|\geq r}\frac{|1-|z|^{-\ga}|^p}{|r^2e_d-z|^{d+sp}}|z|^{\be}\,dz=:\meas\int_{0}^{1}r^{2(d+\al)-1}F(r)\,dr. 
\end{equation}
It suffices now to justify the convergence of the above integral, as $r\rightarrow 0^+$ and $r\rightarrow 1^-$. For $r=1$, the value $F(1)$ is finite, because the integral
$$
F(1)=\int_{|z|\geq 1}\frac{||z|^{\gamma}-1|^p}{|z-e_d|^{d+sp}}|z|^{\be-p\gamma}\,dz
$$
is clearly convergent for $|z|\rightarrow\infty$ (since $\be<sp)$ and for $z\rightarrow e_d$. In the latter case we use the bound
$$
|z|^{\ga}-1=\ga\int_{1}^{|z|}\xi^{\ga-1}\,d\xi\leq\ga \max\{1,|z|^{\ga-1}\}\left(|z|-1\right)\leq\ga \max\{1,|z|^{\ga-1}\}|z-e_d|,
$$
which yields
$$
\frac{||z|^{\gamma}-1|^p}{|z-e_d|^{d+sp}}|z|^{\be-p\ga}\lesssim\frac{1}{|z-e_d|^{d-(1-s)p}},\quad\text{ $z$ close to $e_d$},
$$
and the latter is integrable.

For $r$ close to $0$, we observe that $|z|\geq r$ implies $|z-r^2e_d|\geq |z|-r^2\geq (1-r)|z|$, which gives
$$
F(r)\leq (1-r)^{-d-sp}\int_{|z|\geq r}\frac{|1-|z|^{-\ga}|^p}{|z|^{d+sp-\be}}\,dz.
$$
When $r\rightarrow 0^+$, the latter integral behaves like
$$
\int_{\{|z|\geq r\}}\frac{dz}{|z|^{p\ga+d+sp-\be}}\approx\int_{\{|z|\geq r\}}\frac{dz}{|z|^{2d+\al}}\approx r^{-d-\al}.
$$
Therefore, the integral in \eqref{I} is finite at $r=0$, which is guaranteed by the assumption $\al>-d$.

Overall, summarizing all the computations above, we obtain that there exist positive constants $C_1$, $C_2$ such that the Hardy difference of $u_{\varepsilon}$ is bounded as follows,
$$
0<[u]^p_{W^{s,p}_{\al,\be}(\R^d)}-\mathcal{C}\int_{\R^d}\frac{|u(x)|^p}{|x|^{sp-\al-\be}}\,dx\leq C_1+C_2\frac{\mathcal{C_{\varepsilon}-\mathcal{C}}}{\varepsilon}.
$$
The latter is bounded for $\varepsilon\rightarrow 0^+$, since the function 
$$
t\mapsto\int_{B^c}\int_{B^c}\frac{||x|^{-t}-|y|^{-t}|^p}{|x-y|^{d+sp}}|x|^{\al}|y|^{\be}\,dy\,dx
$$
is differentiable at $t=\ga$. Therefore, using \eqref{p/q}, we have for some positive constant $C=C(d,s,p,\al,\be)$, that
$$
0<C_0\leq\Psi(u_{\varepsilon})\leq C\varepsilon^{p/q}\rightarrow 0,
$$
as $\varepsilon\rightarrow 0^+$. This gives a contradiction and in turn implies that the Hardy--Sobolev--Maz'ya inequality \eqref{HSM} cannot hold with a positive constant $C_0$.

We now move to the case, when $\ga<0$, which is equivalent to $\al+\be<sp-d$. We observe that the substitution $x\mapsto T(x)=x/|x|^2$, $y\mapsto T(y)= y/|y|^2$ transforms the functional $\Psi$ from \eqref{psi} into
$$
\Psi(u_T)=\frac{\displaystyle\int_{\R^d}\int_{\R^d}\frac{|u_T(x)-u_T(y)|^p}{|x-y|^{d+sp}}|x|^{\al'}|y|^{\be'}\,dy\,dx-\mathcal{C}\int_{\R^d}\frac{|u_T(x)|^p}{|x|^{sp-\al'-\be'}}\,dx}{\left(\displaystyle\int_{\R^d}|u_T(x)|^q|x|^{q(\al'+\be')/p}\,dx\right)^{p/q}},
$$
where $u_T(x)=u(T(x))$, $\al'=sp-\al-d$, $\be'=sp-\be-d$. It is straightforward to check that $\mathcal{C}(d,s,p,\al,\be)=\mathcal{C}(d,s,p,\al',\be')$, we only need to ensure that $\al',\be',\al'+\be'\in(-d,sp)$. We clearly have $\al',\be'\in(-d,sp)$, because $\al,\be\in(-d,sp)$. Moreover, since $\al'+\be'=2(sp-d)-(\al+\be)$, we have $\al'+\be'<sp$ if and only if $\al+\be>sp-2d$, but this is true, since $\al+\be>-d>sp-2d$ by the assumption $sp<d$. Also, it holds $\al'+\be'>-d$ if and only if $\al+\be<2sp-d$ and this is also true, since we assume that $\al+\be<sp-d$.

We also have
$$
\ga'=\frac{d+\al'+\be'-sp}{p}=-\frac{d+\al+\be-sp}{p}=-\ga>0,
$$
therefore, we may reduce the case of negative $\ga$ to the previous situation. That ends the proof.
\end{proof}

\end{document}